\documentclass[10pt]{amsart}
\usepackage{amsmath,amssymb,amsthm,amscd}
\usepackage{graphicx}
\usepackage{tikz}
\usepackage{color}
\numberwithin{equation}{section}
\usepackage[plainpages=false]{hyperref}
\newtheorem{prop}{Proposition}[section]
\newtheorem{thm}[prop]{Theorem}
\newtheorem{lem}[prop]{Lemma}
\newtheorem{coro}[prop]{Corollary}
\newtheorem{rem}[prop]{Remark}
\newtheorem{defi}[prop]{Definition}

\def\begeq{\begin{equation}}
\def\endeq{\end{equation}}

\begin{document}
\title[Weighted K-stability of $\mathbb Q$-Fano spherical varieties]
{Weighted K-stability of $\mathbb Q$-Fano spherical varieties}
\author[Yan Li and ZhenYe Li]{Yan Li$^{*1}$, ZhenYe Li$^{*2}$ and Feng Wang$^{*3}$}

\address{$^{*1}$School of Mathematics and Statistics, Beijing Institute of Technology, Beijing, 100081, China.}
\address{$^{*2}$College of Mathematics and Physics, Beijing University of Chemical Technology, Beijing, 100029, China.}
\address{$^{*3}$School of Mathematical Sciences, Zhejiang University, Hangzhou, 310013, China.}
\email{liyan.kitai@yandex.ru,\ \ \ lizhenye@pku.edu.cn,\ \ \ wfmath@zju.edu.cn}

\thanks {$^{*1}$Partially supported by NSFC Grant 12101043 and the Beijing Institute of Technology Research Fund Program for Young Scholars.}
\thanks {$^{*3}$Partially supported by NSFC Grants 11971423 and 12031017.}

\subjclass[2000]{Primary: 14L30, 14M17; Secondary: 14D06, 53C30}

\keywords{Spherical variety, equivariant test configurations, weighted K-stability.}

\begin{abstract}
Let $G$ be a connected, complex reductive Lie group and $X$ a $\mathbb Q$-Fano $G$-spherical variety. In this paper we compute the weighed non-Archimedean functionals of a  $G$-equivariant normal test configurations of $X$ via combinatory data. Also we define a modified Futaki invariant with respect to the weight $g$, and give an expression in terms of intersection numbers. Finally we show the equivalence of different notations of stability and gives a stability criterion on $\mathbb Q$-Fano spherical varieties, which is also a criterion of existence of K\"ahler-Ricci $g$-solitons.
\end{abstract}
\maketitle
\tableofcontents

\section{Introduction}

Let $X$ be an $n$-dimensional projective variety, $D$ an effective divisor such that $K_X +D$ is $\mathbb Q$-Cartier. Let $\Theta\in2\pi c_1(B)$ be a closed positive $(1, 1)$-current for a $\mathbb Q$-line bundle $B$. Assume that $L = -(K_X + D)-B$ is ample. Suppose that $T_\mathbb R\cong (S^1)^r$ be a real torus of rank $r$ whose complexification $T\cong(\mathbb C^*)^r$, which acts effectively and holomorphically on $X$ and preserves the divisor $D$. We further assume that $B$ is also $T$-linearized so that $L = -(K_X + D)-B$ is also $T$-linearized. If the $T$-action is further Hamiltonian, then for any K\"ahler form $\omega_0\in2\pi c_1(L)$, we have the moment map
$${\mathbf m}_{\omega_0}:X\to\Delta\subset\mathfrak t^*.$$
The image $\Delta$ of th moment map is a convex polytope, which is in fact independent with the choice of $\omega_0$.  Let $g$ be any smooth positive function defined on $\Delta$. For any $x\in X$, set $g_{\omega_0}(x):= g({\mathbf m}_{\omega_0}(x))$. For any K\"ahler potential $\phi\in\mathcal E^1_{T_\mathbb R}(\omega_0)$, the space of $T_\mathbb R$-invariant K\"ahler potentials with finite energy (cf. \cite[Definition 2.30]{Han-Li-KRS}), one can define ${\mathbf m}_{\omega_\phi}(\cdot)$ and $g_{\omega_\phi}(\cdot)$ for $\omega_\phi=\omega_0+\sqrt{-1}\partial\bar\partial\phi$ in a same way. Following \cite{Han-Li-KRS}, we say that a metric $\omega_\phi$ with $\phi\in\mathcal E^1_{T_\mathbb R}(\omega_0)$ is a \emph{generalized K\"ahler-Ricci soliton (KR $g$-soliton)} if
\begin{align}\label{g-KRS-eq}
{\rm Ric}(\omega_\phi)=\omega_\phi+[D]+\Theta+\sqrt{-1}\partial\bar\partial\ln g_{\omega_\phi}.
\end{align}
Note that the existence of a KR $g$-soliton implies that the pair $ (X, D+\Theta)$ is klt.

Denote by ${\rm Aut}(X,D)$ the automorphism group of $(X,D)$ and ${\rm Aut}_{T_\mathbb R}(X,D,\Theta)$ its connected subgroup which preserves $\Theta$ and commutes with $T_\mathbb R$. It is proved in \cite[Theorem 1.7]{Han-Li-KRS} that:
\begin{itemize}
\item If $(X,D+\Theta,T_\mathbb R)$ is $\mathfrak G$-uniformly $g$-Ding-stable over $(T\times\mathfrak G)$-equivariant test configurations for a connected reductive group $\mathfrak G$, then $X$ admits a KR $g$-soliton;
\item If $X$ admits a KR $g$-soliton and $\mathfrak G$ a connected reductive subgroup of ${\rm Aut}_{T_\mathbb R}(X,D,\Theta)$ that contains a maximal torus of ${\rm Aut}_{T_\mathbb R}(X,D,\Theta)$. Then $(X,D+\Theta,T_\mathbb R)$ is $\mathfrak G$-uniformly $g$-Ding-stable over $(T\times\mathfrak G)$-equivariant test configurations.
\end{itemize}

The $\mathfrak G$-uniformly $g$-Ding stability of $(X,D+\Theta,{T_\mathbb R})$ is defined in terms of $g$-weighted non-Archimedean Ding functional and $g$-weighted non-Archimedean J-functional (cf. \cite[Section 5]{Han-Li-KRS}). This stability implies properness of certain modified Ding functional (cf. \cite[Section 6]{Han-Li-KRS}) and the existence can be derived via variational methods. When $g=1$ and $\Theta=D=0$, \eqref{g-KRS-eq} reduces to the usual K\"ahler-Einstein problem, and the $\mathfrak G$-uniformly $g$-Ding stability reduces to the usual $\mathfrak G$-uniformly Ding stability (cf. \cite{Boucksom-Hisamoto-Jonsson, Li19}). In the following we recall a non-trivial example. Assume that $\xi$ is a holomorphic vector field on $M$ which generates a rank $r$ torus $T$-action on $M$. Denote by $\xi_1,...,\xi_r$ the generators of $T$. Then by a suitable choice of the generators, $\xi=\sum_{A=1}c_A\xi_A$ for constants $c_1,...,c_r\in\mathbb R$. In fact, the soliton vector field $\xi$ can be uniquely determined by \cite{TZ-2}. Recall the K\"ahler metric $\omega_\phi\in2\pi c_1(L)$. Let $\theta_A(\omega_\phi)$ be the Hamiltonian of $\xi_A$ with respect to $\omega_0$,
$$\iota_{\xi_A}\omega_\phi=\sqrt{-1}\bar\partial\theta_A(\omega_\phi),~A=1,...,r.$$
Then $$\mathbf{m}_{\omega_\phi}(x)=(\theta_1(\omega_\phi),...,\omega_r(\omega_\phi)).$$
Take $D=\Theta=0$ and $$g_{\omega_\phi}=e^{\theta_\xi(\omega_\phi)}=e^{\sum_{A=1}c_A\theta_{\xi_A}(\omega_\phi)}.$$ Then \eqref{g-KRS-eq} reduces to the K\"ahler-Ricci soliton equation.  \cite[Theorem 1.7]{Han-Li-KRS} then gives a existence criterion of the K\"ahler-Ricci solitons. Note that if we choose $\{\lambda_A:=\theta_{\xi_A}(\omega_\phi)\}_{A=1}^r$ as the coordinates on $\Delta$, then the corresponding $g(y)=e^{\sum_{A=1}^rc_A\lambda_A}$ is exponential of an affine function on $\Delta$. The Mabuchi metric problem can also be treated in this framework (cf. \cite{LZ, Yao}).

On the other hand, Wang-Zhou-Zhu \cite{WZZ} introduced the modified Futaki invariant and modified K-stability for the K\"ahler-Ricci soliton problem. They defined the modified Futaki invariant of a test configuration via weighted total weights (cf. \cite[Sections 1-2]{WZZ}), which generalized the modified Futaki invariant of vector fields defined in \cite{TZ-2}. Moreover, they showed that this invariant has an integration-expression on the central fibre for any special test configuration. As an application, they showed that any toric Fano variety is modified K-stable, which then implies that the modified K-energy is proper.

Motivated by the works cited above, in this paper we consider the general KR $g$-soliton problem when $D=\Theta=0$ and $g$ a general continuous function. As in \cite{WZZ} we define $g$-modified Futaki invariants of a $\mathbb Q$-Fano variety (see Section 2.2 below). Concerning the existence of KR $g$-solitons with $D=\Theta=0$ on $\mathbb Q$-Fano varieties with at most klt singularities, we are now having various notations of ($g$-modified) stability: The $g$-modified K-stability which is defined according to the sign of the ($g$-modified) Futaki invariant (see Definition \ref{K-stab-def-Fut} below) introduced in the sense of \cite{Do,WZZ}; The $g$-K-stability which is defined according to the sign of the ($g$-weighted) non-Archimedean Mabuchi functional (see Definition \ref{K-stab-def-Fut} below), which is introduced in the sense of \cite{Boucksom-Hisamoto-Jonsson}; The $g$-Ding-stability which is defined according to the sign of the ($g$-weighted) non-Archimedean Ding functionals, which was introduced in \cite{Han-Li-KRS}.
We showed the existence of $g$-modified Futaki invariant when $g$ is a non-negative polynomial and compare the above notations of ($g$-)stability. We will show their equivalence on $\mathbb Q$-Fano varieties:
\begin{thm}\label{K-stab-def-equiv-thm}
Let $X$ be a $\mathbb Q$-Fano variety and $T\subset{\rm Aut}(X)$ be complex torus. Assume that $-K_X$ is $T$-linearized through some fixed lifting. Let $\Delta$ be the moment polytope with respect to this lifting and $g$ a polynomial function on $\Delta$.\footnote{In fact, this assumption is independent with the choice of a lifting. See Section 2.1.3 below.} Then for any $T$-equivariant normal test configuration of $(X,-K_X)$, the $g$-modified Futaki invariant exists. Furthermore, if $g\geq0$, then $X$ is $g$-modified K-polystable in the sense of Definition \ref{K-stab-def-Fut} if and only if it is $T$-equivariantly $g$-K-polystable in the sense of Definition \ref{K-stab-def-MNA}.
\end{thm}
The existence part will be proved in Theorem \ref{Fut-g-smooth-thm}. In fact we prove an intersection formula of the $g$-modified Futaki invariant for polynoimal $g$. Compareing with the intersection formula of the non-Archimedean Mabuchi functionals, Theorem \ref{K-stab-def-equiv-thm} is then a consequence of Proposition \ref{K-stab-equiv} below.

On the other hand, let $G$ be a connect, complex reductive group and $X$ a $\mathbb Q$-Fano $G$-spherical varieties. We will show in Section 5 the existence of $g$-modified Futaki invariant for any smooth $g$. With the help of Proposition \ref{mod-futaki} and Corollary \ref{FNA-prop} below, we can significantly strength Theorem \ref{K-stab-def-equiv-thm} to the following:
\begin{thm}\label{K-stab-def-equiv-spherical-thm}
Let $X$ be a $\mathbb Q$-Fano $G$-spherical variety and $T\subset{\rm Aut}^0_G(X)$ be a complex torus that commutes with $G$. Let $\Delta$ be the moment polytope of the $T$-action and $g$ an arbitrary smooth function on $\Delta$. Then for any $G$-equivariant normal test configuration of $(X,-K_X)$, the $g$-modified Futaki invariant exists. Moreover, $X$ is $G$-equivariantly $g$-modified K-polystable in the sense of Definition \ref{K-stab-def-Fut} if and only if it is $G$-equivariantly $g$-K-polystable in the sense of Definition \ref{K-stab-def-MNA}.
\end{thm}
The existence part will be showed in Proposition \ref{mod-futaki}, and the equivalence part in Corollary \ref{FNA-prop}. Finally we get the following stability/existence criterion for $\mathbb Q$-Fano spherical varieties, which is a generalization of \cite[Theorem A]{Del3}:

\begin{thm}\label{stab-criterion-Fano}
Let $X$ be a $\mathbb Q$-Fano $G$-spherical variety. Set $\mathbb G:=G\times{\rm Aut}_G^0(X)$. Then the following are equivalent:
\begin{itemize}
\item[(1)] The $g$-weighted barycenter
\begin{align}\label{bar-condition}
{\mathbf b}_g(\Delta_+):=\frac1{V_g}\int_{\Delta_+}\lambda g\pi(\lambda)d\lambda\in\kappa_P+{\rm RelInt}(-\mathcal V(G/H))^\vee;
\end{align}
\item[(2)] $X$ is $\mathbb G$-uniformly $g$-Ding stable;
\item[(3)] $X$ is $\mathbb G$-uniformly $g$-K-stable;
\item[(4)] $X$ is $\mathbb G$-equivariantly $g$-K-polystable.
\end{itemize}
\end{thm}

\begin{rem}
By \cite{Han-Li-KRS} (cf. \cite[Theorems 1.6 and 6.3]{Han-Li-KRS}), one can conclude that \eqref{bar-condition} holds if and only if $X$ admits a K\"ahler-Ricci $g$-soliton. Consequently, \eqref{bar-condition} implies that $X$ is $g$-Ding/K-polystable (regardless group actions).
\end{rem}

Our method is to direct computing the non-Archimedean functionals by using the intersection formula in \cite[Section 18]{Timashev-book}, and the Futaki invariant using the asymptotic expression of the total weights. In particular, we get an inequality of the non-Archimedean Mabuchi functional and the Futaki invariant for non-negative, smooth, compactly supported weight $g$ (see Corollary \ref{FNA-prop} below). 

The paper is organized as follows: Section 2 studies $g$-weighted non-Archimedean functionals, $g$-weighted Futaki invariant on general $\mathbb Q$-Fano varieties. We introduce various notations of $g$-modified/weighted stabilities. Especially in Section 2.1 we give change of lifting formulas for Archimedean functionals, which will play an important role in computations. In Section 2.2 we define the $g$-modified Futaki invariant. We also study the change of lifting formula of the $g$-modified Futaki invariants and give a formula of the $g$-modified Futaki invariant for polynomial $g$ via some intersection numbers. In Section 2.3 we prove Theorem \ref{K-stab-def-equiv-thm}. Sections 3-6 are devoted to the KR $g$-soliton problem on $\mathbb Q$-Fano $G$-spherical varieties. In Section 3 we recall preliminaries of spherical varieties. In particular we study the fibre product construction introduced in \cite[Section 2.1]{Han-Li-KRS} on polarized spherical varieties. In Section 4 we compute the $g$-weighted non-Archimedean functionals. In Section 5 we study the $g$-modified Futaki invariant. In Section 6 we prove the stability criterion Theorem \ref{stab-criterion-Fano}. In the Appendix we collect useful Lemmas.

\subsection*{Acknowledgement} The authors would sincerely thank Professor Chi Li for many helpful discussions.

\section{The notations of stability}
The various notations of stability of a polarized variety are usually stated in terms of test configurations. Let $(X,L)$ be a polarized variety. A normal test configuration of $(X,L)$ consists of the following data:
\begin{itemize}
\item A normal variety $\mathcal X$ with a $\mathbb C^*$-action;
\item An ample line bundle $\mathcal L$ on $\mathcal X$;
\item A $\mathbb C^*$-equivariant flat morphism $\tilde\pi:(\mathcal X,\mathcal L)\to\mathbb C$ so that the fibre $(\mathcal X_t,\mathcal L_t)$ over $t\in\mathbb C^*$ is isomorphic to $(X,m_0L)$ for some $m_0\in\mathbb N_+$.
\end{itemize}
For our later us, in the following we compactify $(\mathcal X,\mathcal L)$ to a family over $\mathbb P^1$ by adding a trivial fibre $(X,m_0L)$ at $\infty\in\mathbb P^1$. Alternatively, we glue $(\mathcal X,\mathcal L)$ with $(X,m_0L)\times\mathbb C\cong(X,m_0L)\times(\mathbb C^*\cup\{\infty\})$ along the common part $(X,m_0L)\times\mathbb C^*$. From now on, by $(\mathcal X,\mathcal L)$ we always refer to the compactified family. Also, by resolution of singularity, we can assume that $(\mathcal X,\mathcal L)$ is dominating, that is, there is a $\mathbb C^*$-equivariant birational morphism $\rho:\mathcal X\to X\times\mathbb P^1$.

A test configuration $(\mathcal X,\mathcal L)$ is called \emph{product} if $(\mathcal X,\mathcal L)\cong(X,m_0L)\times\mathbb P^1$. Let $\mathfrak G$ be a group acts on $(X,L)$. A test configuration $(\mathcal X,\mathcal L)$ is called \emph{$\mathfrak G$-equivariant} if $\mathfrak G$ acts on $(\mathcal X,\mathcal L)$ and $\tilde\pi$ is $\mathfrak G$-invariant. That is, $\mathfrak G$ acts on each fibre.

In the remaining part of Section 2.1, we always assume that $X$ is $\mathbb Q$-Fano and take $L=-K_X$.

\subsection{The $g$-weighted non-Archimedean functionals}

\subsubsection{Definition under the canonical lifting}
Suppose that $X$ is an $n$-dimensional $\mathbb Q$-Fano variety with an $r$-dimensional torus $T$-action. Then the anticanonical line bundle $L=-K_X$ is automatically $T$-linearized with a canonical lifting of $T$-action on it. Let $(\mathcal X,\mathcal L)$ be a normal test configuration of $(X,-K_X)$ with central fibre $\mathcal X_0$. Clearly, there is an induced lifting of the $T$-action on $\mathcal L$. Denote by $V=(-K_X)^n$ the volume of $X$. Under the canonical lifting of the $T$-action, the usual non-Archimedean functionals ${\rm E}^{\rm NA}(\cdot),{\rm I}^{\rm NA}(\cdot), {\rm J}^{\rm NA}(\cdot),{\rm H}^{\rm NA}(\cdot), {\rm M}^{\rm NA}(\cdot), {\rm L}^{\rm NA}(\cdot),$ and ${\rm D}^{\rm NA}(\cdot)$ are defined by (cf. \cite{Boucksom-Hisamoto-Jonsson}):

\begin{align*}
{\rm E}^{\rm NA}(\mathcal X, \mathcal L)=&\frac1{(n+1)V}(\mathcal L^{n+1}),\\
{\rm J}^{\rm NA}(\mathcal X, \mathcal L)=&\frac1{V}\mathcal L\cdot L_{\mathbb P^1}^n-{\rm E}^{\rm NA}(\mathcal X, \mathcal L),\\
{\rm L}^{\rm NA}(\mathcal X, \mathcal L)=&{\rm lct}_{(\mathcal X,-(\mathcal L+K_{\mathcal X/\mathbb {P}^1}))}(\mathcal X_0)-1;\\
{\rm D}^{\rm NA}(\mathcal X, \mathcal L)=&{\rm L}^{\rm NA}(\mathcal X, \mathcal L)-{\rm E}^{\rm NA}(\mathcal X, \mathcal L),\\
{\rm M}^{\rm NA}(\mathcal X, \mathcal L)=&\frac1{Vn!}\mathcal L^n\cdot K^{\log}_{\mathcal X/\mathbb P^1}+n{\rm E}^{\rm NA}(\mathcal X, \mathcal L),
\end{align*}
where
\begin{align}\label{K-log-def}
K^{\log}_{\mathcal X/\mathbb P^1}:=K_{\mathcal X}+\mathcal X_{0,{\rm red}}-\tilde\pi^*\{\infty\},
\end{align}
is understood as a Weil divisor. The Weil divisor $K_{\mathcal X}$ can be realized as following: Let $\mathcal X_{\rm reg}$ be the regular locus of $\mathcal X$. Since $\mathcal X$ is normal, the singular locus $\mathcal X\setminus\mathcal X_{\rm reg}$ as codimension at least $2$. On $\mathcal X_{\rm reg}$ there is a section $\tilde s$ of $K_{\mathcal X}|_{\mathcal X_{\rm reg}}$ which defines a divisor $\tilde{\mathfrak d}_0$ in $\mathcal X_{\rm reg}$ and $K_{\mathcal X}$ in \eqref{K-log-def} is the closure of $\tilde{\mathfrak d}_0$ in $\mathcal X$.

Then we recall the definition of $g$-weighted non-Archimedean functionals under this canonical lifting. This was first formulated in \cite{Han-Li-KRS}. The functionals are defined by $g$-weighted intersection numbers (cf. \cite[Sections 5, 10]{Han-Li-KRS}).

\emph{Step-1. $g$ is a monomial.} In this case, the $g$-weighted intersection numbers are defined as intersection numbers of line bundles over a fibre product variety $(X^{[\mathbf{k}]},L^{[\mathbf{k}]})$. Let us recall this construction of \cite[pp.7-8]{Han-Li-KRS}. Let $X$ be a $\mathbb Q$-Fano variety which admits a rank $r$ torus $T$-action. Also assume that $T$ acts on $-K_X$ is $T$-linearized. Suppose that $\{\xi_A\}_{A=1}^r\subset\mathfrak t$ is a set of generators of $T$. Then $\{\xi_A^*\}_{A=1}^r$ is a basis of $\mathfrak t^*$. Let $\{\theta_A\}_{A=1}^r$ be the coordinates of $\mathfrak t^*$ under this basis. Suppose that
\begin{align}\label{monom-g}
g(\theta_1,...,\theta_r)=\prod_{A=1}^r\theta_A^{k_A}(=:\theta^{\mathbf k}),~\text{for}~k_1,...,k_r\in\mathbb N_+,
\end{align}
where $\mathbf{k}=(k_1,...,k_r)\in\mathbb N^r$. Denote 
$\mathbb C^{[\mathbf{k}]+\mathbf{1}}=\prod_{A=1}^r\mathbb C^{k_A+1}$. Consider the action of a torus $\mathbb T\cong T$ on $X\times \mathbb C^{[\mathbf{k}]+\mathbf{1}}$,
$$\vartheta(x;z^{(A),i_A}):=(\iota(\vartheta)x;\vartheta_Az^{(A,i_A)}),\vartheta\in\mathbb T$$
where $\iota:\mathbb T\to T$ is the isomorphism between $\mathbb T$ and $T$,  $\{z^{(A),i_A}\}_{i_A=0}^{k_A}$ are the coordinates on $\mathbb C^{k_A+1}$, and we write $(x;z^{(A),i_A})$ in short of $(x;z^{(1),i_0},...,z^{1,k_1},...,$ $z^{(r),i_0},...,z^{r,k_r})$. Take $L=-K_X$, define
$$(X^{[\mathbf{k}]},L^{[\mathbf{k}]}):=(X,L)\times(\mathbb C^{\mathbf{k}+\mathbf{1}}\setminus\{O\})/\mathbb T.$$
Then $(X^{[\mathbf{k}]},L^{[\mathbf{k}]})$ is a bundle over $\mathbb P^{[\mathbf{k}]}=\mathbb P^{k_1}\times...\times\mathbb P^{k_r}$,
$${\rm pr}^{[\mathbf{k}]}:X^{[\mathbf{k}]}\to\mathbb P^{[\mathbf{k}]},$$
where is the projection. Moreover, each fibre of $(X^{[\mathbf{k}]},L^{[\mathbf{k}]})$ is isomorphic to $(X,L)$.

Set $$V_g:=\int_Xg_\phi\frac{\omega_\phi^n}{n!},$$
$\mathbf{k}!:=k_1!...k_r!$, and $|\mathbf{k}|:=k_1+...+k_r.$ Also, for any $T$-equivariant normal test configuration $(\mathcal X,\mathcal L)$, define\footnote{Our convention differs from \cite{Han-Li-KRS} by a factor $\frac1{V_g}$. Also, we normalized the Fubini-Study metric on the $m$-dimensional projective space as $\int_{\mathbb P^{m}}{\omega_{{\rm FS}}^{m}}=1$ instead of $m!$.}
\begin{align}
{\rm E}_g^{\rm NA}(\mathcal X, \mathcal L)=&\frac{\mathbf{k}!}{(n+|\mathbf{k}|+1)!V_g}(\mathcal L^{[\mathbf{k}]})^{n+|\mathbf{k}|+1},\label{E-g-mono}\\
{\rm I}_g^{\rm NA}(\mathcal X, \mathcal L)=&\frac1{L^n}\mathcal L\cdot L_{\mathbb P^1}^n-\frac{\mathbf{k}!}{(n+|\mathbf{k}|)!}(\mathcal L-\rho^*L_{\mathbb P^1})^{[\mathbf {k}]}(\mathcal L^{[\mathbf {k}]})^{n+|\mathbf{k}|},\label{I-g-mono}\\
{\rm J}_g^{\rm NA}(\mathcal X, \mathcal L)=&\frac1{L^n}\mathcal L\cdot L_{\mathbb P^1}^n-{\rm E}_g^{\rm NA}(\mathcal X, \mathcal L),\label{J-g-mono}\\
{\rm H}^{\rm NA}(\mathcal X, \mathcal L)=&\frac{\mathbf{k}!}{V_g(n+|\mathbf k|)!}(\mathcal L^{[\mathbf{k}]})^{n+k}\cdot ((K^{\log}_{\mathcal X/\mathbb P^1})^{[\mathbf{k}]}-(\rho^{[\mathbf{k}]})^*(K^{\log}_{X\times\mathbb P^1/\mathbb P^1})^{[\mathbf{k}]}),\label{H-g-mono}\\
{\rm D}_g^{\rm NA}(\mathcal X, \mathcal L)=&{\rm L}^{\rm NA}(\mathcal X, \mathcal L)-{\rm E}_g^{\rm NA}(\mathcal X, \mathcal L),\label{D-g-mono}\\
{\rm M}_g^{\rm NA}(\mathcal X, \mathcal L)=&{\rm H}^{\rm NA}(\mathcal X, \mathcal L)-{\rm I}_g^{\rm NA}(\mathcal X, \mathcal L)+{\rm J}_g^{\rm NA}(\mathcal X, \mathcal L)\notag\\
=&\frac{\mathbf{k}!}{V_g(n+|\mathbf k|)!}(\mathcal L^{[\mathbf{k}]})^{n+k}\cdot (K^{\log}_{\mathcal X/\mathbb P^1})^{[\mathbf{k}]}+(n+|\mathbf {k}|){\rm E}_g^{\rm NA}(\mathcal X, \mathcal L).\label{M-g-mono}
\end{align}
Here the Weil divisor $(K^{\log}_{\mathcal X/\mathbb P^1})^{[\mathbf{k}]}$ is defined as following: Recall the section $\tilde s$ of $K_{\mathcal X}|_{\mathcal X_{\rm reg}}$. Suppose that it has $T$-weight $\tilde\lambda_0$ under the canonical lifting of the $T$-action. Let $\overline{\tilde{\mathfrak d}_0}$ be the closure of $\tilde{\mathfrak d}_0$ in $\mathcal X$. Then it is a $T$-invariant Weil divisor of $\mathcal X$ and $\overline{\tilde{\mathfrak d}_0}^{[\mathbf k]}:=\overline{\tilde{\mathfrak d}_0}\times(\mathbb C^{\mathbf{k}+\mathbf{1}}\setminus\{O\})/\mathbb T$ is a Weil divisor of $\mathcal X^{[\mathbf k]}$. Then we define
\begin{align}\label{K_X/P-log-k-Weil-div}
K^{\log}_{\mathcal X/\mathbb P^1}:=\overline{\tilde{\mathfrak d}_0}^{[\mathbf k]}+\sum_{A=1}^r{\rm pr}_A^*\mathcal O_{\mathbb P^{k_A}}(\tilde\lambda_{0A})+\mathcal X_{0,{\rm red}}^{[\mathbf k]}-(\tilde\pi^*\{\infty\})^{[\mathbf k]}.
\end{align}

Note that since $g(\lambda)=\prod_{A=1}^r\lambda_A^{k_A}$ is a monomial of degree $|\mathbf{k}|$, it holds\footnote{We would like to thank Professor Chi Li for pointing us this relation.}
$$\frac{|\mathbf{k}|}{(n+|\mathbf{k}|+1)!}(\mathcal L^{[\mathbf{k}]})^{n+|\mathbf{k}|+1}=V_{\lambda(\nabla g)}{\rm E}^{\rm NA}_{\lambda(\nabla g)}(\mathcal L),$$
where $y(\nabla g(\lambda))=\sum_{A=1}^r\lambda_A\frac\partial{\partial \lambda_A}g(\lambda)$.

\emph{Step-2. $g$ is a polynomial.} 
Let
\begin{align}\label{g-polynomial}
g=\sum_{\mathbf{k}}a_{\mathbf{k}}\theta^{\mathbf{k}},~a_{\mathbf{k}}\in\mathbb C
\end{align}
be a polynomial. 
Also, for ${\rm F}\in\{\rm E,J,M\}$, define
\begin{align}\label{NNA-general-g-def}
{\rm F}_g^{\rm NA}(\mathcal X,\mathcal L):=\frac1{V_g}&\sum_{\mathbf{k}}a_{\mathbf{k}}V_{\theta^{\mathbf{k}}}\cdot{\rm F}^{\rm NA}_{\theta^{\mathbf{k}}}(\mathcal X,\mathcal L).
\end{align}

\emph{Step-3. $g$ is a continuous function.}
Let $g$ be a general $C^0$-function on $\Delta$. Then there is a sequence of polynomials $\{g_k\}_{k=1}^{+\infty}$ so that $g_k$ converges to $g$ uniformly on $\Delta$. Define
\begin{align*}
V_g=&\lim_{k\to+\infty} V_{g_k},\\
{\rm F}^{\rm NA}_g(\mathcal X,\mathcal L)=&\lim_{k\to+\infty} {\rm F}^{\rm NA}_{g_k}(\mathcal X,\mathcal L),~{\rm F}\in\{\rm E,I,J,H,D,M\}.
\end{align*}
It is proved by \cite[Sections 5 and 10]{Han-Li-KRS} that none of the above limits depends on the choice of  $\{g_k\}_{k=1}^{+\infty}$. Hence they are well-defined.

An important property proved in \cite{Han-Li-KRS} is that the non-Archimedean functionals defined above satisfies the slope formula, which means that they are the slope of the corresponding Archimedean functionals at infinity (cf. \cite[Propositions 5.8 and 10.8]{Han-Li-KRS}).

\subsubsection{Change of the lifting}
For our later use, we will consider the expression of the non-Archimedean functionals ${\rm F}^{\rm NA}(\cdot)$ under an arbitrary lifting of the $T$-action on $L=-K_X$. Given a $\mathbb Q$-Fano variety $X$ and a torus $T\subset{\rm Aut}(X)$. Let us fix a lifting $\sigma$ of the $T$-action on $L$. Denote by $\Delta\subset\mathfrak t^*$ the corresponding polytope and choose a coordinate $y_1,...,y_r$ on it. For a $T$-equivariant normal test configuration $(\mathcal X,\mathcal L)$ of $(X,L)$, we will denote by  $(\mathcal X,\mathcal L^\sigma)$ to emphasize the lifting $\sigma$. Also we omit $\sigma$ when we refer to the canonical lifting. In general, $L^{\sigma[\mathbf k]}$ and $L^{[\mathbf k]}$ ($\mathcal L^{\sigma[\mathbf k]}$ and $\mathcal L^{[\mathbf k]}$, resp.) are different line bundles on $X^{[\mathbf k]}$ ($\mathcal X^{[\mathbf k]}$, resp.).

For a test configuration $(\mathcal X,\mathcal L)$, by resolution of singularity, we can assume that it is dominating. That is, there is a $\mathbb C^*$-equivariant birational morphism $\rho:\mathcal X\to X\times\mathbb P^1$. Then $\mathcal L=\rho^*L+D$ for some $\mathbb Q$-Cartier divisor $D$, and $(\mathcal X,\mathcal L)$ induces a canonical non-Archimedean metric on $L^{\rm NA}$. The non-Archimedean functionals can be derived from the corresponding Archi-medean ones by taking slope at infinity. From the construction \emph{Steps-2, 3} in Section 2.1, it suffices to compute ${\rm F}^{\rm NA}(\mathcal X,\mathcal L^\sigma)$ when $g$ is a monomial given by \eqref{monom-g}.

\begin{prop}\label{F-NA-general-lifting}
Suppose that $g$ is a monomial given by \eqref{monom-g}. Then for a general lifting $\sigma$ of the $T$-action on $L+-K_X$, it holds
\begin{align}
{\rm E}_g^{\rm NA}(\mathcal X,\mathcal L^\sigma)=&\frac{\mathbf{k}!}{(n+|\mathbf k|+1)!}(\mathcal L^{\sigma[\mathbf k]})^{n+|\mathbf k|+1},\label{E-NA-general-lifting}
\end{align}
\begin{align}
{\rm J}_g^{\rm NA}(\mathcal X,\mathcal L^\sigma)=&\frac{\mathbf{k}!}{(n+|\mathbf k|)!}\mathcal L^{\sigma[\mathbf k]}(-\rho^{[\mathbf k]*}K_{X\times\mathbb P^1/\mathbb P^1}^{\sigma[\mathbf k]})^{n+|\mathbf k|}-{\rm E}_g^{\rm NA}(\mathcal X,\mathcal L^\sigma),\label{J-NA-general-lifting}
\end{align}
\begin{align}
{\rm I}_g^{\rm NA}(\mathcal X,\mathcal L^\sigma)=&\frac{\mathbf{k}!}{(n+|\mathbf k|)!}\mathcal L^{\sigma[\mathbf k]}(-\rho^{[\mathbf k]*}K_{X\times\mathbb P^1/\mathbb P^1}^{\sigma[\mathbf k]})^{n+|\mathbf k|}\notag\\
&-\frac{\mathbf{k}!}{(n+|\mathbf k|)!}(\mathcal L^{\sigma[\mathbf k]}+\rho^{[\mathbf k]*}K_{X\times\mathbb P^1/\mathbb P^1}^{\sigma[\mathbf k]})(\mathcal L^{\sigma[\mathbf k]})^{n+|\mathbf k|},\label{I-NA-general-lifting}
\end{align}
\begin{align}
{\rm H}_g^{\rm NA}(\mathcal X,\mathcal L^\sigma)=&\frac{\mathbf{k}!}{(n+|\mathbf k|)!}(K_{\mathcal X/\mathbb P^1}^{\log[\mathbf k]}-\rho^{[\mathbf k]*}K_{X\times\mathbb P^1/\mathbb P^1}^{[\mathbf k]})(\mathcal L^{\sigma[\mathbf k]})^{n+|\mathbf k|},\label{H-NA-general-lifting}
\end{align}
\begin{align}
{\rm M}_g^{\rm NA}(\mathcal X,\mathcal L^\sigma)=&{\rm H}_g^{\rm NA}(\mathcal X,\mathcal L^\sigma)-{\rm I}_g^{\rm NA}(\mathcal X,\mathcal L^\sigma)+{\rm J}_g^{\rm NA}(\mathcal X,\mathcal L^\sigma)\notag\\
=&\frac{\mathbf{k}!}{(n+|\mathbf k|)!}(K_{\mathcal X/\mathbb P^1}^{\log[\mathbf k]})(\mathcal L^{\sigma[\mathbf k]})^{n+|\mathbf k|}+\frac{\mathbf{k}!}{(n+|\mathbf k|)!}(\rho^{[\mathbf k]*}K^{\sigma[\mathbf k]}_{X\times\mathbb P^1/\mathbb P^1}\notag\\
&-\rho^{[\mathbf k]*}K^{[\mathbf k]}_{X\times\mathbb P^1/\mathbb P^1})(\mathcal L^{\sigma[\mathbf k]})^{n+|\mathbf k|}+(n+|\mathbf k|){\rm E}_g^{\rm NA}(\mathcal X,\mathcal L^\sigma).\label{M-NA-general-lifting}
\end{align}

\end{prop}

\begin{proof}
Following \cite[Section 2.1]{Han-Li-KRS}, we have for any $\omega_0$-psh functions $\phi_0,\phi=\phi_1$ with finite Monge-Amp\`ere energy and path $\{\phi_t\}_{t\in[0,1]}$ joining them, it holds
\begin{align*}
{\rm E}_g(\phi)=&\frac1{V_g}\int_0^1\int_X\dot\phi_tg(\mathbf{m}_{\omega_{\phi_t}})\frac{\omega_{\phi_t}^n}{n!}\notag\\
=&\frac{\mathbf{k}!}{V_g}\int_0^1\int_{X^{[\mathbf k]}}\dot\phi^{[\mathbf k]}_t\frac{(\omega_{\phi_t}+\sum_{A=1}^r\omega_{{\rm FS};A})^{n+|\mathbf k|}}{(n+|\mathbf k|)!}\notag\\
=&\frac{\mathbf{k}!}{(n+|\mathbf k|+1)!V_g}\sum_{l=1}^{n+|\mathbf k|}\int_{X^{[\mathbf k]}}(\phi^{[\mathbf k]}-\phi^{[\mathbf k]}_0)(\omega_{\phi_0}+\sum_{A=1}^r\omega_{{\rm FS};A})^{l}\wedge\notag\\ &\wedge(\omega_{\phi}+\sum_{A=1}^r\omega_{{\rm FS};A})^{n+|\mathbf k|-l},
\end{align*}
where $\omega_{{\rm FS}}$ denotes the Fubini-Study metric on $\mathbb P^{m}$, 
and $\omega_{{\rm FS};A}$ denotes $\omega_{{\rm FS}}$ on $\mathbb P^{k_A}$.
Note that $\phi^{[\mathbf k]}$ is a metric on $L^{\sigma[\mathbf k]}$. Taking slope at infinity in the above relation we get \eqref{E-NA-general-lifting}.

For \eqref{J-NA-general-lifting}, we have
\begin{align*}
{\rm J}_g(\phi)=&\frac1{V_g}\int_X(\phi-\phi_0)g(\mathbf{m}_{\omega_{\phi_t}})\frac{\omega_{\phi_0}^n}{n!}-{\rm E}_g(\phi)\notag\\
=&\frac{\mathbf{k}!}{V_g}\int_{X^{[\mathbf k]}}(\phi^{[\mathbf k]}-\phi^{[\mathbf k]}_0)\frac{(\omega_{\phi_0}+\sum_{A=1}^r\omega_{{\rm FS};A})^{n+|\mathbf k|}}{(n+|\mathbf k|)!}-{\rm E}_g(\phi).
\end{align*}
Thus
\begin{align}\label{J-NA-general-mono}
{\rm J}_g^{\rm NA}(\mathcal X,\mathcal L^\sigma)=&\frac{\mathbf{k}!}{(n+|\mathbf k|)!}(\mathcal L^{\sigma[\mathbf k]}+\rho^{[\mathbf k]*}K_{X\times\mathbb P^1/\mathbb P^1}^{\sigma[\mathbf k]})(-\rho^{[\mathbf k]*}K_{X\times\mathbb P^1/\mathbb P^1}^{\sigma[\mathbf k]})^{n+|\mathbf k|}\notag\\&-{\rm E}_g^{\rm NA}(\mathcal X,\mathcal L^\sigma),
\end{align}
and \eqref{J-NA-general-lifting} follows from the fact that
$$(-\rho^{[\mathbf k]*}K_{X\times\mathbb P^1/\mathbb P^1}^{\sigma[\mathbf k]})^{n+|\mathbf k|+1}=0.$$
Similarly, we get \eqref{I-NA-general-lifting}.

Finally we prove \eqref{H-NA-general-lifting}. Rewrite $\omega_0=\sqrt{-1}\partial\bar\partial\varphi_0$ so that it is the curvature of the Hermitian metric $e^{-\varphi_0}$ on $-K_X$. Then $d\mu_0=e^{-\varphi_0}$ is a globally defined measure on $X$.
Recall \cite[Section 10.1]{Han-Li-KRS},
\begin{align*}
{\rm H}_g(\phi)=&\frac1{V_g}\int_X\ln\left(\frac{g(\mathbf{m}_{\omega_{0}}){\omega_{0}^n}}{n!d\mu_0}\right)g(\mathbf{m}_{\omega_0})\frac{\omega_0^n}{n!}\\
=&\frac{\mathbf{k}!}{V_g}\int_{X^{[\mathbf k]}}\ln\left(\frac{{\omega_0^n}}{n!d\mu_0}\right)\frac{(\omega_0+\sum_{A=1}^r\omega_{{\rm FS};A})^{n+|\mathbf k|}}{(n+|\mathbf k|)!}\notag\\&+\frac{\mathbf{k}!}{V_g}\int_{X^{[\mathbf k]}}\ln g(\mathbf{m}_{\omega_0})\frac{(\omega_0+\sum_{A=1}^r\omega_{{\rm FS};A})^{n+|\mathbf k|}}{(n+|\mathbf k|)!}.
\end{align*}
On the other hand, as showed in \cite[Section 10.1]{Han-Li-KRS},
\begin{align*}
&\int_{X^{[\mathbf k]}}\ln\left(\frac{{\omega_0^n}}{n!d\mu_0}\right)\frac{(\omega_0+\sum_{A=1}^r\omega_{{\rm FS};A})^{n+|\mathbf k|}}{(n+|\mathbf k|)!}\notag\\
=&\langle(\ln(\omega_0^n))^{[\mathbf k]},\phi^{[\mathbf k]},...,\phi^{[\mathbf k]}\rangle_{X^{[\mathbf k]}}-\langle(\ln(e^{-\varphi_0}))^{[\mathbf k]},\phi_0^{[\mathbf k]},...,\phi_0^{[\mathbf k]}\rangle_{X^{[\mathbf k]}},
\end{align*}
where $\langle...\rangle_{X^{[\mathbf k]}}$ denotes the metric on the Deligne pair. Note that $\ln(\omega_0^n))^{[\mathbf k]}$ and $\ln(e^{-\varphi_0})^{[\mathbf k]}$ are always metrics on line bundles with respective to the canonical lifting. We get \eqref{H-NA-general-lifting}. The relation \eqref{M-NA-general-lifting} then follows from \eqref{E-NA-general-lifting}-\eqref{H-NA-general-lifting}.

\end{proof}
Clearly, the second term in \eqref{M-NA-general-lifting} vanishes if $\sigma$ is the canonical lifting.

\subsubsection{Invariance of the non-Archimedean functionals}
Suppose that there are two different liftings $\sigma_1,\sigma_2$ of the $T$-action on $-K_X$ so that
\begin{align}\label{sigma2-sigma1}
\sigma_2=\sigma_1+\chi
\end{align}
for some $T$-character $\chi$. Then the corresponding moment maps
$$\mathbf{m}_2=\mathbf{m}_1+\chi.$$
Clearly, the corresponding moment polytopes $\Delta_2=\Delta_1+\chi$.

Consider the equation
\begin{align}\label{g-krs-eq-1}
g_{\omega_\phi}^{(1)}\omega_\phi^n=n!e^{-\phi}\omega_0^n,
\end{align}
where $g_{\omega_\phi}^{(i)}:=g^{(i)}\circ\mathbf {m}_1$ with $g^{(i)}:\Delta_1\to\mathbb R$. The equation \eqref{g-krs-eq-1} can be changed into
\begin{align}\label{g-krs-eq-2}
g_{\omega_\phi}^{(2)}\omega_\phi^n=n!e^{-\phi}\omega_0^n,
\end{align}
where $g_{\omega_\phi}^{(2)}:=g^{(2)}\circ\mathbf {m}_1$ with
\begin{align}\label{change-of-g-def}
g^{(2)}=g^{(1)}(\mathbf{m}_1-\chi):\Delta_2\to\mathbb R.
\end{align}
The equations \eqref{g-krs-eq-1}, \eqref{g-krs-eq-2} are associated to the weighted non-Archimedean Mabuchi functionals ${\rm M}^{\rm NA}_{g^{(1)}}(\cdot)$ and ${\rm M}^{\rm NA}_{g^{(2)}}(\cdot)$, respectively.

Suppose that $(\mathcal X,\mathcal L)$ is a $T$-equivariant normal test configuration of $(X,-K_X)$. We denote by  $(\mathcal X,\mathcal L^{\sigma_i}),~i=1,2$ for $(\mathcal X,\mathcal L)$ with the $T$-action via $\sigma_i,~i=1,2$, respectively. In the following, we will show the invariance of the non-Archimedean functionals:
\begin{prop}\label{change-of-NA-func}
Suppose that $g^{(1)}\in C^\infty(\Delta_1)$. Then
$${\rm F}^{\rm NA}_{g^{(1)}}(\mathcal X,\mathcal L^{\sigma_1})={\rm F}^{\rm NA}_{g^{(2)}}(\mathcal X,\mathcal L^{\sigma_2}),~{\rm F\in\{E,I,J,H,M,D\}}.$$
\end{prop}

\begin{proof}
As in the previsous section, it suffices to prove the Proposition when $g^{(1)}$ is given by \eqref{monom-g}. For convenience, we may choose $\sigma_1$ to be the canonical lifting.

We first show the Proposition for $\rm {F=E}$. Denote by $\mathbf{k}=(k_1,...,k_r)$. For convenience, for another $\mathbf{i}=(i_1,...,i_r)$, we write $\mathbf i\leq\mathbf k$ if each $i_A\leq k_A$. Also, set $\mathbf 0=(0,...,0)$. By \eqref{sigma2-sigma1},
\begin{align}\label{change-of-L}
(\mathcal L^{\sigma_1})^{[\mathbf k]}=(\mathcal L^{\sigma_2})^{[\mathbf k]}-\sum_{A=1}^r{\rm pr}_A^*\mathcal O_{\mathbb P^{k_A}}(\chi_A),
\end{align}
where ${\rm pr}_A:\mathcal X^{[\mathbf k]}\to\mathbb P^{k_A}$ is the projection to the $A$-th factor in $\mathbb P^{[\mathbf k]}$. Thus
\begin{align*}
((\mathcal L^{\sigma_1})^{[\mathbf k]})^{n+|\mathbf k|+1}=&((\mathcal L^{\sigma_2})^{[\mathbf k]}-\sum_{A=1}^r{\rm pr}_A^*\mathcal O_{\mathbb P^{k_A}}(\chi_A))^{[\mathbf k]})\notag\\
=&\sum_{\mathbf 0\leq\mathbf i\leq\mathbf k}\frac{(n+|\mathbf k|+1)!}{(n+|\mathbf k|+1-|\mathbf i|)!\mathbf i!}(-1)^{|\mathbf i|}(\mathcal L^{\sigma_2[\mathbf k]})^{n+|\mathbf k|+1-|\mathbf i|}\times\notag\\
&\times\prod_{A=1}^r({\rm pr}_A^*\mathcal O_{\mathbb P^{k_A}}(\chi_A))^{i_A}.
\end{align*}
Note that
\begin{align}\label{intersection-L-O(1)}
(\mathcal L^{\sigma_2[\mathbf k]})^{n+|\mathbf k|+1-|\mathbf i|}\prod_{A=1}^r({\rm pr}_A^*\mathcal O_{\mathbb P^{k_A}}(1))^{i_A}=(\mathcal L^{\sigma_2[\mathbf k-\mathbf i]})^{n+|\mathbf k|+1-|\mathbf i|}.
\end{align}
We get
\begin{align*}
{\mathbf{k}!}((\mathcal L^{\sigma_1})^{[\mathbf k]})^{n+|\mathbf k|+1}=&\sum_{\mathbf0\leq\mathbf i\leq\mathbf k}\frac{(n+|\mathbf k|+1)!}{(n+|\mathbf {k-i}|+1)!}C_{\mathbf k!}^{\mathbf i!}(-1)^{|\mathbf i|}\chi^{\mathbf i}{(\mathbf{k}-\mathbf{i})!}(\mathcal L^{\sigma_2[\mathbf {k-i}]})^{n+|\mathbf {k-i}|+1}.
\end{align*}
Here we write $C_{\mathbf k!}^{\mathbf i!}=\prod_{A=1}^rC_{k_A}^{i_A}$ for short.

On the other hand,
\begin{align}\label{change-of-g}
g^{(2)}(\lambda)=g^{(1)}(\lambda-\chi)=&\prod_{A=1}^r(\lambda_A-\chi_A)^{k_A}=\sum_{\mathbf0\leq\mathbf i\leq\mathbf k}C_{\mathbf k!}^{\mathbf i!}(-1)^{|\mathbf i|}\prod_{A=1}^r\chi^{\mathbf i}\lambda^{\mathbf{k-i}}.
\end{align}
Combining with \eqref{E-g-mono}, we see that
\begin{align*}
{\rm E}^{\rm NA}_{g^{(2)}}(\mathcal X,\mathcal L^{\sigma_2})={\rm E}^{\rm NA}_{g^{(1)}}(\mathcal X,\mathcal L^{\sigma_1}).
\end{align*}
The case $\rm{F=D}$ can be checked by using the above relation and \eqref{D-g-mono}.

For the case $\rm{F=J}$, recall the fact
\begin{align}\label{change-of-rhoL}
\rho^{[\mathbf k]*}K_{X\times\mathbb P^1/\mathbb P^1}^{\sigma_1[\mathbf{k}]}=\rho^{[\mathbf k]*}K_{X\times\mathbb P^1/\mathbb P^1}^{\sigma_2[\mathbf{k}]}+\sum_{A=1}^r{\rm pr}_A^*\mathcal O_{\mathbb P^{k_A}}(\chi_A).
\end{align}
Combining with \eqref{change-of-L},
$$(\mathcal L^{\sigma_1[\mathbf k]}+\rho^{[\mathbf k]*}K_{X\times\mathbb P^1/\mathbb P^1}^{\sigma_1[\mathbf k]})=(\mathcal L^{\sigma_2[\mathbf k]}+\rho^{[\mathbf k]*}K_{X\times\mathbb P^1/\mathbb P^1}^{\sigma_2[\mathbf k]}).$$
By \eqref{J-NA-general-mono} we proved the Proposition for  $\rm{F=J}$. The case $\rm{F=I}$ can be showed in a same way.

Finally we consider the case $\rm{F=M}$. Recall that we assume $\sigma_1$ is the canonical lifting. As in the case $\rm{F=E}$, by \eqref{change-of-L}-\eqref{intersection-L-O(1)}, we can show that

\begin{align}\label{change-of-first-term}
&{\mathbf{k}!}(\mathcal L^{\sigma_1})^{[\mathbf{k}]})^{n+|\mathbf k|}(K^{\log}_{\mathcal X/\mathbb P^1})^{[\mathbf{k}]}\notag\\
=&((\mathcal L^{\sigma_2})^{[\mathbf k]}-\sum_{A=1}^r{\rm pr}_A^*\mathcal O_{\mathbb P^{k_A}}(\chi_A))^{n+|\mathbf{k}|}\cdot (K^{\log}_{\mathcal X/\mathbb P^1})^{[\mathbf{k}]}\notag\\
=&\sum_{\mathbf0\leq\mathbf{i}\leq\mathbf k}\frac{(n+|\mathbf k|)!}{(n+|\mathbf{k-i}|)!}C_{\mathbf k!}^{\mathbf i!}(-\chi)^{\mathbf i}(\mathbf{k}-\mathbf{i})!(K^{\log}_{\mathcal X/\mathbb P^1})^{\mathbf{k}-\mathbf{i}}((\mathcal L^{\sigma_2})^{\mathbf{k}-\mathbf{i}})^{\cdot(n+|\mathbf{k}-\mathbf{i}|)}.
\end{align}

Similarly,
\begin{align}\label{change-of-second-term}
&|\mathbf k|{\rm E}^{\rm NA}_{g^{(1)}}(\mathcal X,\mathcal L^{\sigma_1})\notag\\
=&\sum_{\mathbf0\leq\mathbf{i}\leq\mathbf k}(-\chi)^{\mathbf i}|\mathbf k-\mathbf i|C_{\mathbf k}^{\mathbf i}{\rm E}^{\rm NA}_{\lambda^{\mathbf k-\mathbf i}}(\mathcal X,\mathcal L^{\sigma_2})+\sum_{\mathbf0\leq\mathbf{i}\leq\mathbf k}(-\chi)^{\mathbf i}|\mathbf i|C_{\mathbf k}^{\mathbf i}{\rm E}^{\rm NA}_{\lambda^{\mathbf k-\mathbf i}}(\mathcal X,\mathcal L^{\sigma_2})\notag\\
=&\sum_{\mathbf0\leq\mathbf{i}\leq\mathbf k}(-\chi)^{\mathbf i}|\mathbf k-\mathbf i|C_{\mathbf k}^{\mathbf i}{\rm E}^{\rm NA}_{\lambda^{\mathbf k-\mathbf i}}(\mathcal X,\mathcal L^{\sigma_2})-{\rm E}^{\rm NA}_{\chi(\nabla(\lambda-\chi)^{\mathbf k})}(\mathcal X,\mathcal L^{\sigma_2}),
\end{align}
where $$\chi(\nabla(\lambda-\chi)^{\mathbf k}=\sum_{A=1}^r\chi_A\frac\partial{\partial\lambda_A}(\lambda-\chi)^{\mathbf k}.$$

On the other hand, by \eqref{change-of-rhoL},
\begin{align*}
&\frac{\mathbf j!}{(n+|\mathbf j|)!}(\rho^{[\mathbf j]*}K_{X\times\mathbb P^1/\mathbb P^1}^{\sigma_2[\mathbf{j}]}-\rho^{[\mathbf j]*}K_{X\times\mathbb P^1/\mathbb P^1}^{\sigma_1[\mathbf{j}]})(\mathcal L^{\sigma_2[\mathbf j]})^{n+|\mathbf j|}\\
=&-\frac{\mathbf j!}{(n+|\mathbf j|)!}(\sum_{A=1}^r{\rm pr}_A^*\mathcal O_{\mathbb P^{k_A}}(\chi_A))(\mathcal L^{\sigma_1[\mathbf j]})^{n+|\mathbf j|}.\\
=&-\frac1{(n+|\mathbf j|)!}\sum_{A=1}^r\chi_A\frac{\mathbf j!}{\mathbf j'_A!}{\mathbf j_A'!}(\mathcal L^{\sigma_2[\mathbf j_A']})^{n+|\mathbf j'_A|+1}.\\
=&-\sum_{A=1}^r\chi_Aj_A{\rm E}_{\lambda^{\mathbf j_A'}}^{\rm NA}(\mathcal X,\mathcal L^{\sigma_2})=-{\rm E}_{\sum_{A=1}^r\chi_A\frac{\partial}{\partial \lambda_A}\lambda^{\mathbf j}}^{\rm NA}(\mathcal X,\mathcal L^{\sigma_2}),
\end{align*}
where used the fact $|\mathbf j|=|\mathbf j_A'|+1$. Thus,
\begin{align*}
&\sum_{\mathbf j\leq\mathbf k}\frac{C_{\mathbf k}^{\mathbf j}(-\chi)^{\mathbf j}{\mathbf j!}}{(n+|\mathbf j|)!}(\rho^{[\mathbf j]*}K_{X\times\mathbb P^1/\mathbb P^1}^{\sigma_2[\mathbf{j}]}-\rho^{[\mathbf j]*}K_{X\times\mathbb P^1/\mathbb P^1}^{\sigma_1[\mathbf{j}]})(\mathcal L^{\sigma_2[\mathbf j]})^{n+|\mathbf j|}
=-{\rm E}^{\rm NA}_{\chi(\nabla(\lambda-\chi)^{\mathbf k})}(\mathcal X,\mathcal L^{\sigma_2}).
\end{align*}
Plugging this equation into \eqref{change-of-second-term}, we have
\begin{align*}
&|\mathbf k|{\rm E}^{\rm NA}_{g^{(1)}}(\mathcal X,\mathcal L^{\sigma_1})\notag\\
=&\sum_{\mathbf 0\leq\mathbf j\leq\mathbf k}C_{\mathbf k}^{\mathbf j}(-\chi)^{\mathbf j}\frac{\mathbf j!}{(n+|\mathbf j|)!}(\rho^{[\mathbf j]*}K_{X\times\mathbb P^1/\mathbb P^1}^{\sigma_2[\mathbf{j}]}-\rho^{[\mathbf j]*}K_{X\times\mathbb P^1/\mathbb P^1}^{\sigma_1[\mathbf{j}]})(\mathcal L^{\sigma_2[\mathbf j]})^{n+|\mathbf j|}\notag\\
&+\sum_{\mathbf0\leq\mathbf{i}\leq\mathbf k}(-\chi)^{\mathbf i}|\mathbf k-\mathbf i|C_{\mathbf k}^{\mathbf i}{\rm E}^{\rm NA}_{\lambda^{\mathbf k-\mathbf i}}(\mathcal X,\mathcal L^{\sigma_2}).
\end{align*}
Combining with \eqref{change-of-first-term} and \eqref{change-of-g} we get the Proposition for $\rm{F=M}$.

\end{proof}

\subsection{The $g$-modified Futaki invariant}
There is also a geometric way to construct test configuration of a $\mathbb Q$-Fano variety $X$ (cf. \cite[Section 2.2]{Han-Li} and \cite[Section 1]{WZZ}). Suppose that there is a Kodaira embedding of $X$ by $|-{n_0}K_X|$ for some $n_0\in\mathbb N_+$,
$$i:X\to\mathbb P({\rm H}^0(X,-{n_0}K_X))=:\mathbb P^{N-1}.$$
Choose a vector $\Lambda\in\mathfrak{psl}_N(\mathbb C)$ so that $\Lambda$ generates a rank $1$ torus of ${\rm PSL}_N(\mathbb C)$. Then it defines a test configuration $(\mathcal X,\mathcal L)$ via
$$\mathcal X_{t}:=\exp(z\Lambda)\cdot i(X),~t=e^z\in\mathbb C^*$$
and
$$\mathcal L|_{\mathcal X_t}:=\mathcal O_{\mathbb P^{N-1}}(1)|_{\mathcal X_t}.$$
Also define $\mathcal X_0:=\lim_{t\to0}\mathcal X_t$ as the limit of algebraic cycle, and $\mathcal L_0:=\mathcal L_{\mathcal X_0}$. Indeed, any test configuration can be realized in this way.

Without loss of generality we assume that $n_0=1$. Otherwise, we replace $-K_X$ by $-n_0K_X$. Suppose that $\xi$ is a (real) holomorphic vector field on $X$ so that $T=\overline{\exp(t\xi)}$. We also assume that $\Lambda$ commutes with $\xi$ so that the test configuration $(\mathcal X,\mathcal L)$ is $T$-equivariant. We can also lift $\xi$ to an element of $\mathfrak{psl}_N(\mathbb C)$. Choose a basis $\{e_p\}_{p=1}^{{\rm h}^0(X,-K_X)}$ of ${\rm H}^0(\mathcal X_0,-\mathcal L_0)$ so that each $e_p$ is a common eigenvector of both the $\exp(t\xi)$- and $\exp(t\Lambda)$-actions. Denote by $\{e^{\xi_p^k}\}_{p=1}^{{\rm h}^0(X,-kK_X)}$ and $\{e^{\Lambda_p^k}\}_{p=1}^{{\rm h}^0(X,-kK_X)}$ the eigenvalues of the canonical lifting of the $\exp(t\xi)$- and $\exp(t\Lambda)$-actions on ${\rm H}^0(\mathcal X_0,-k\mathcal L_0)$, respectively. Here we use the fact that ${\rm h}^0(X,-kK_X)={\rm h}^0(\mathcal X_0,-k\mathcal L_0)$. Fix a background K\"ahler metric $\omega_0\in2\pi c_1(X)$ and denote by $\theta_\xi$ a potential of $\xi$ with respect to $\omega_0$. Suppose that $T$ has generators $\{\xi_A\}_{A=1}^r$. As in \cite[Section 1]{WZZ}, for a $C^1$-function $g$ defined on some interval and $g_{\omega_0}=g(\theta_{\xi_1},...,\theta_{\xi_r})$, define
\begin{align}
S^{(g)}_{1|k}(\mathcal X,\mathcal L):=&\sum_{p=1}^{{\rm h}^0(X,-kK_X)}g(\frac{\xi_{1|p}^k}k,...,\frac{\xi_{r|p}^k}k)\Lambda_p^k,\label{S-1k-def}\\
S^{(g)}_{2|k}(\mathcal X,\mathcal L):=&\frac12\sum_{p=1}^{{\rm h}^0(X,-kK_X)}\sum_{A=1}^r\frac{\partial g}{\partial \theta_{\xi_A}}(\frac{\xi_{1|p}^k}k,...,\frac{\xi_{r|p}^k}k)\frac{\xi_{A|p}^k}k \frac{\Lambda_p^k}k.\label{S-2k-def}
\end{align}
Once it holds the formal asymptotic expression,
\begin{align}\label{S1-S2}
\frac{S^{(g)}_{2|k}-S^{(g)}_{1|k}}{k{\rm h}^0(X,-kK_X)}(\mathcal X,\mathcal L)=:F_0(\mathcal X,\mathcal L)+F_1(\mathcal X,\mathcal L)k^{-1}+O(k^{-2}),~k\to+\infty,
\end{align}
we can define the $g$-modified Futaki invariant in a similar way of \cite[Section 1]{WZZ},
\begin{defi}\label{g-mod-Fu-def}
Suppose that \eqref{S1-S2} holds. Then the $g$-modified Futaki invariant of $(\mathcal X,\mathcal L)$ is defined as
$${\rm Fut}_g(\mathcal X,\mathcal L):=F_1(\mathcal X,\mathcal L).$$
\end{defi}

It was showed by \cite[Section 1]{WZZ} that \eqref{S1-S2} holds for $g=e^{\theta_\xi}$ on a Fano manifold. That is, the modified Futaki invariant in  Definition \ref{g-mod-Fu-def} for K\"ahler-Ricci solitons is well-defined. At the beginning of Section 5, we will show that \eqref{S1-S2} holds for an equivariant test configuration of a polarized spherical variety with $g\in C^\infty(\Delta)$. Hence the  $g$-modified Futaki invariant is defined. Indeed, we can show the $g$-modified Futaki invariant of an equivariant test configuration can be expressed as intersection numbers of some line bundle (see Lemma \ref{FNA-prop} below).

\subsubsection{Change of lifting}
As in Section 2.1.2, we consider the expression of the Futaki invariant under a general lifting of the $T$-action on $-K_X$. Denote by $\sigma_0$ the canonical lifting and $\sigma$ so that there is a $T$-character $\chi$ such that
\begin{align}\label{sigma-sigma0}
\sigma=\sigma_0+\chi
\end{align}
Then the corresponding weights
$$\xi_{A|p}^k(\sigma)=\xi_{A|p}^k(\sigma_0)+\chi(\xi_{A|p}^k),~\forall p\in\mathbb N_+~\text{and}~k=1,...,{\rm h}^0(X,-kK_X).$$
Also denote by $g^{(\sigma)}(\cdot),~g^{(\sigma_0)}(\cdot)$ the weight functions on each moment polytope, respectively. Then define
\begin{align}
S^{(g^{(\sigma)})}_{1|k}(\mathcal X,\mathcal L):=&\sum_{p=1}^{{\rm h}^0(X,-kK_X)}g^{(\sigma)}(\frac{\xi_{1|p}^k(\sigma)}k,...,\frac{\xi_{r|p}^k(\sigma)}k)\Lambda_p^k,\label{S-1k-def-change}\\
S^{(g^{(\sigma)})}_{2|k}(\mathcal X,\mathcal L):=&\frac12\sum_{p=1}^{{\rm h}^0(X,-kK_X)}\sum_{A=1}^r\frac{\partial g^{(\sigma)}}{\partial \theta_{\xi_A}}(\frac{\xi_{1|p}^k(\sigma)}k,...,\frac{\xi_{r|p}^k(\sigma)}k)\frac{\xi_{A|p}^k(\sigma_0)}k \frac{\Lambda_p^k}k,\label{S-2k-def-change}
\end{align}
and the $g$-weighted Futaki invariant is defined as Definition \ref{g-mod-Fu-def}. In can be checked that ${\rm Fut}_{g^{(\sigma)}}(\cdot)$ derived from  \eqref{S-1k-def-change}-\eqref{S-2k-def-change} coincides with that of  ${\rm Fut}_{g^{(\sigma_0)}}(\cdot)$ derived from \eqref{S-1k-def}-\eqref{S-2k-def}.

\subsubsection{The case of polynomial $g$}

Let $X$ be a $\mathbb Q$-Fano variety. Without loss of generality we can assume that $L=-K_X$ is very ample so that $X$ is embedded in a projective space by $|-K_X|$. Suppose that $T\subset{\rm Aut}(X)$ is a complex torus of rank $r$, with $\{\xi_A\}_{A=1}^r$ the generators. Also suppose that $(\mathcal X,\mathcal L)$ is a test configuration of $(X,-K_X)$ constructed at the beginning of Section 2.2, with $\Lambda$ the generator of the corresponding $\mathbb C^*$-action. In the following we will show \eqref{S1-S2} holds for polynomial $g$. We will adopt the argument of \cite[Section 3]{Boucksom-Hisamoto-Jonsson}.


\begin{thm}\label{Fut-g-smooth-thm}
Let $X$ be a $\mathbb Q$-Fano variety and $T\subset{\rm Aut}(X)$ acts on $-K_X$ through some fixed lifting $\sigma$. Suppose that $g$ is a polynomial function on the moment polytope $\Delta$ of the lifting $\sigma$. Then for any $T$-equivariant normal test configuration $(\mathcal X,\mathcal L)$ of $(X,-K_X)$, \eqref{S1-S2} holds. Moreover, the $g$-modified Futaki invariant of  $(\mathcal X,\mathcal L)$ is given as following:
\begin{itemize}
\item[(1)] When $g$ is a monomial given by \eqref{monom-g},
 \begin{align}\label{Fut-g-monomial}
{\rm Fut}_g(\mathcal X,\mathcal L)=&\frac{\mathbf{k}!}{(n+|\mathbf k|)!V}(K_{\mathcal X/\mathbb P^1}^{[\mathbf k]})(\mathcal L^{\sigma[\mathbf k]})^{n+|\mathbf k|}+\frac{\mathbf{k}!}{(n+|\mathbf k|)!V}(\rho^{[\mathbf k]*}K^{\sigma[\mathbf k]}_{X\times\mathbb P^1/\mathbb P^1}\notag\\
&-\rho^{[\mathbf k]*}K^{[\mathbf k]}_{X\times\mathbb P^1/\mathbb P^1})(\mathcal L^{\sigma[\mathbf k]})^{n+|\mathbf k|}+\frac{V_g}V (n+|\mathbf k|){\rm E}_g^{\rm NA}(\mathcal X,\mathcal L^\sigma);
\end{align}
\item[(2)] When $g$ is a monomial given by \eqref{g-polynomial},
\begin{align}\label{Fut-general-g-def}
{\rm Fut}_g(\mathcal X,\mathcal L):=&\sum_{\mathbf{k}}a_{\mathbf{k}}{\rm Fut}_{\lambda^{\mathbf{k}}}(\mathcal X,\mathcal L).
\end{align}
\end{itemize}
\end{thm}

\begin{proof}
We will prove the Proposition using the argument of \cite[Section 3]{Boucksom-Hisamoto-Jonsson}. It suffices to prove it when $\sigma$ is the canonical lifting. In view of the change of lifting formula \eqref{S-1k-def-change}-\eqref{S-2k-def-change}, we may fix a lifting $\sigma'$ of the $T$-action on $-K_X$ so that the corresponding moment polytope $\Delta'$ lies in the first quadrant. In fact, we may even assume that $\mathcal L^{\sigma'[{\mathbf k}]}$ is ample. We first deal with the case when $g'$ is a monomial on $\Delta'$ given by \eqref{monom-g}. Note that if $\sigma'=\sigma_0+\mu$, then $g'(\cdot+\mu)$ is the corresponding weight on the moment polytope $\Delta=\Delta'-\mu$ of the canonical lifting $\sigma_0$.

For a normal test configuration ${\rm pr}:\mathcal X\to\mathbb P^1,$ set ${\mathbf k}:=(k_1,...,k_r)$ and consider $(\mathcal X^{[{\mathbf k}]}, \mathcal L^{[{\mathbf k}]})$. 
Denote by ${\rm pr}^{[{\mathbf k}]}:\mathcal X^{[{\mathbf k}]}\to\mathbb P^{[{\mathbf k}]}$ the projection. We also consider the variety $\mathbb P^1$ with trivial $\mathbb T$-action. Then $(\mathbb P^1)^{[\mathbf{k}]}=\mathbb P^1\times\mathbb P^{[\mathbf{k}]}$, and we have the following diagram:
\begin{align*}
\mathcal X&\times(\mathbb C^{\mathbf{k}+\mathbf{1}}{\setminus}\{O\}) \overset{/\mathbb T}\longrightarrow \mathcal X^{[\mathbf k]}\\
{{\rm pr}\times {\rm Id}}&\downarrow \textcolor[rgb]{1.00,1.00,1.00}{--------} \downarrow{\rm pr}^{[{\mathbf{k}}]}\\
\mathbb P^1&\times(\mathbb C^{\mathbf{k}+\mathbf{1}}{\setminus}\{O\})\overset{/\mathbb T}\longrightarrow \mathbb P^1\times\mathbb P^{[\mathbf k]}\overset{{\rm pr}^0}\longrightarrow\mathbb P^1
\end{align*}
The diagram commutes since the $T_0:=\exp(t\Lambda)$- and $\mathbb T$-actions commute. For any $m\in\mathbb N_+$, denote by ${\rm H}^0(\mathcal X_0,-m\mathcal L_0)_{(\eta;\lambda)}$ the sections in ${\rm H}^0(\mathcal X_0,-m\mathcal L_0)$ so that $T$ acts on it through character $\eta=(\eta_1,...,\eta_r)\in\mathfrak X(T)\cong\mathbb Z^r$ and $\exp(t\Lambda)$ acts through character $\lambda\in\mathbb Z$. 
Also we have $(m\mathcal L)^{\sigma'[{\mathbf k}]}=m\mathcal L^{\sigma'[{\mathbf k}]}$. Let $\overline{\rm pr}:={\rm pr}^0\circ{\rm pr}^{[\mathbf k]}$. Then using \cite[Proposition 1.3]{Boucksom-Hisamoto-Jonsson},
\begin{align*}
\overline{\rm pr}_*\mathcal O_{\mathcal X^{[{\mathbf k}]}}(m\mathcal L^{\sigma'[{\mathbf k}]})&=\oplus_{\lambda\in\mathbb Z}{\rm H}^0(\mathcal X_0,-m\mathcal L_0)^{(T_0)}_{\lambda}\otimes\mathcal O_{\mathbb P^1}(\lambda),
\end{align*}
recall that here $T_0$ is the $\mathbb C^*$-action of the test configuration which acts trivially on $\mathbb C^{\mathbf{k}+\mathbf{1}}{\setminus}\{O\}$. Clearly,
$${\rm H}^0(\mathcal X_0,-m\mathcal L_0)^{(T_0)}_{\lambda}=\oplus_{\eta\in\mathfrak X(T)}{\rm H}^0(\mathcal X_0,-m\mathcal L_0)^{(T\times T_0)}_{(\eta,\lambda)},$$
since the actions of $T_0$ and $T$ commutes. By Lemma \ref{dim-H0-Lk-lem} in the Appendix,
$$\dim{\rm H}^0(\mathcal X_0,-m\mathcal L_0)^{(T_0)}_{\lambda}=\sum_{\eta\in\mathfrak X(T)}\dim{\rm H}^0(\mathcal X_0,-m\mathcal L_0)^{(T\times T_0)}_{(\eta,\lambda)}C_{\mathbf k+\lambda}^{\mathbf k}.$$
Thus,
\begin{align*}
\chi(\mathbb P^1,\overline{\rm pr}_*\mathcal O_{\mathcal X^{[{\mathbf k}]}}(m\mathcal L^{\sigma'[{\mathbf k}]}))
=\sum_{(\eta;\lambda)\in\mathfrak X(T)\oplus\mathbb Z}\dim{\rm H}^0(\mathcal X,-m\mathcal L)^{(T\times T_0)}_{(\eta;\lambda)}(\lambda+1)C_{\mathbf k+\lambda}^{\mathbf k}.
\end{align*}

Hence, we get
\begin{align}\label{chi(P^k)}
&\chi(\mathbb P^1,\overline{\rm pr}_*\mathcal O_{\mathcal X^{[{\mathbf k}]}}(m\mathcal L^{\sigma'[{\mathbf k}]}))\notag\\
=&\sum_{(\eta;\lambda)\in\mathfrak X(T)\oplus\mathbb Z}(\lambda+1)\dim{\rm H}^0(\mathcal X_0,-m\mathcal L_0)^{(T\times T_0)}_{(\eta;\lambda)}\prod_{A=1}^r\left(\frac{(\eta_A+1)\cdot...\cdot(\eta_A+k_A)}{k_A!}\right).
\end{align}
Also, by the Riemann-Roch formula \cite[Theorem A.1]{Boucksom-Hisamoto-Jonsson},
\begin{align}\label{chi(X^k)}
\chi(\mathcal X^{[{\mathbf k}]},m\mathcal L^{\sigma'[{\mathbf k}]})=&\frac{(\mathcal L^{\sigma'[{\mathbf k}]})^{n+|{\mathbf k}|+1}}{(n+|{\mathbf k}|+1)!}m^{n+|{\mathbf k}|+1}-\frac{K_{\mathcal X^{[{\mathbf k}]}}\cdot(\mathcal L^{\sigma'[{\mathbf k}]})^{n+|{\mathbf k}|}}{2(n+|{\mathbf k}|)!}m^{n+|{\mathbf k}|}\notag\\
&+O(m^{n+|{\mathbf k}|-1}),~m\to+\infty.
\end{align}
On the other hand, since $L$ is ample, by the Leray spectral sequence \cite[Chapter 1, Section 4.2]{Danilov-Itogi},
\begin{align}\label{larey-eq}
\chi(\mathbb P^1,\overline{\rm pr}_*\mathcal O_{\mathcal X^{[{\mathbf k}]}}(m\mathcal L^{[{\mathbf k}]}))=\chi(\mathcal X^{[{\mathbf k}]},m\mathcal L^{[{\mathbf k}]}).
\end{align}
Combining with \eqref{chi(P^k)}-\eqref{chi(X^k)}, we get the sum
$$\sum_{(\eta;\lambda)\in\mathfrak X(T)\oplus\mathbb Z}\lambda\frac{\eta^{\mathbf k}}{\mathbf k!}\dim{\rm H}^0(\mathcal X_0,m\mathcal L_0)^{(T\times T_0)}_{(\eta;\lambda)}$$
is a polynomial function for large $m\gg1$,
\begin{align}\label{first-term-in-S1}
&\sum_{(\eta;\lambda)\in\mathfrak X(T)\oplus\mathbb Z}\lambda\frac{\eta^{\mathbf k}}{\mathbf k!}\dim{\rm H}^0(\mathcal X_0,m\mathcal L_0)^{(T\times T_0)}_{(\eta;\lambda)}\notag\\
=&
\frac{(\mathcal L^{\sigma'[{\mathbf k}]})^{\cdot(n+|{\mathbf k}|+1)}}{(n+|{\mathbf k}|+1)!}m^{n+|{\mathbf k}|+1}
+O(m^{n+|{\mathbf k}|}),~m\to+\infty.
\end{align}
Similarly, denote by ${\mathbf{k}}'_A=(k_1,...,k_A-1,...,k_r)$, it holds
\begin{align}\label{second-term}
&\sum_{(\eta;\lambda)\in\mathfrak X(T)\oplus\mathbb Z}\lambda\frac{\eta^{\mathbf k_A'}}{\mathbf k_A'!}\dim{\rm H}^0(\mathcal X_0,m\mathcal L_0)^{(T\times T_0)}_{(\eta;\lambda)}\notag\\
=&
\frac{(\mathcal L^{\sigma'[{\mathbf{k}}_A']})^{\cdot(n+|{\mathbf k}|)}}{(n+|{\mathbf k}|)!}m^{n+|{\mathbf k}|}
+O(m^{n+|{\mathbf k}|-1}),~m\to+\infty.
\end{align}
Also,
\begin{align}\label{error-term}
&\sum_{(\eta;\lambda)\in\mathfrak X(T)\oplus\mathbb Z}\frac{\eta^{\mathbf k}}{\mathbf k!}\dim{\rm H}^0(\mathcal X_0,m\mathcal L_0)^{(T\times T_0)}_{(\eta;\lambda)}\notag\\
=&
\frac{(\mathcal L_0^{\sigma'[{\mathbf k}]})^{\cdot(n+|{\mathbf k}|)}}{(n+|{\mathbf k}|)!}m^{n+|{\mathbf k}|}
+O(m^{n+|{\mathbf k}|-1}),~m\to+\infty.
\end{align}
Thus, plug \eqref{chi(P^k)}-\eqref{chi(X^k)} and \eqref{first-term-in-S1}-\eqref{error-term} into \eqref{larey-eq}, we get
\begin{align}\label{sum-s1m-pol-pre}
S^{(g')}_{1|m}(\mathcal X,\mathcal L)=&\sum_{(\eta;\lambda)\in\mathfrak X(T)\oplus\mathbb Z}\lambda g'(\frac{\eta_A}{m})\dim{\rm H}^0(\mathcal X_0,m\mathcal L_0)^{(T\times T_0)}_{(\eta;\lambda)}\notag\\
=&{{\mathbf{k}}}!\frac{(\mathcal L^{\sigma'[{\mathbf k}]})^{n+|{\mathbf k}|+1}}{(n+|{\mathbf k}|+1)!}m^{n+1}-{{\mathbf{k}}}!\frac{K_{\mathcal X^{[{\mathbf k}]}}\cdot(\mathcal L^{\sigma'[{\mathbf k}]})^{n+|{\mathbf k}|}}{2(n+|{\mathbf k}|)!}m^{n}\notag\\
&-{{\mathbf{k}}}!\sum_{A=1}^r(k_A+1)\frac{(\mathcal L^{\sigma'[{\mathbf k}'_A]})^{n+|{\mathbf k}|}}{2(n+|{\mathbf k}|)!}m^{n}-{{\mathbf{k}}}!\frac{( \mathcal L_0^{\sigma'[{\mathbf k}]})^{n+|{\mathbf k}|}}{(n+|{\mathbf k}|)!}m^{n}\notag\\
&+O(m^{n+|\bar{\mathbf k}|-2}),~m\to+\infty.
\end{align}

We want to simplify the above equation. The relation of $-K_{X^{[\mathbf k]}}$ and $(-K_X)^{[\mathbf k]}$ is derived in Lemma \ref{-K_{X^k}-lem} in the Appendix. On the other hand, by definition, a $T$-invariant divisor $D$ in ${\rm pr}_A^*\mathcal O_{\mathbb P^{k_A}}(1)$ satisfies
$$(D,\mathcal L^{\sigma'[{\mathbf k}]}|_D)\cong(\mathcal X^{[{\mathbf{k}}'_A]},\mathcal L^{\sigma'[{\mathbf{k}}'_A]}).$$
It then follows
\begin{align}\label{intersection-O(k_A+1)}
{\rm pr}_A^*\mathcal O_{\mathbb P^{k_A}}(k_A+1)\cdot(\mathcal L^{\sigma'[{\mathbf k}]})^{n+|{\mathbf k}|}=(k_A+1)(\mathcal L^{\sigma'[{\mathbf k}'_A]})^{n+|{\mathbf k}|}.
\end{align}

Also, consider the projection ${\rm pr}:\mathcal X\to\mathbb P^1$. Since $-K_{\mathbb P^1}=\mathcal O_{\mathbb P^1}(2)$, it has a divisor $-K_{\mathbb P^1}=2[0]$. 
We have
\begin{align}\label{error-simp}
-({\rm pr}^{[\mathbf{k}]})^*K_{\mathbb P^1}\cdot(\mathcal L^{\sigma'[{\mathbf k}]})^{n+|{\mathbf k}|}=2(\mathcal L_0^{\sigma'[{\mathbf k}]})^{n+|{\mathbf k}|}.
\end{align}

Plugging \eqref{intersection-O(k_A+1)}-\eqref{error-simp} into \eqref{sum-s1m-pol-pre} and using  Lemma \ref{-K_{X^k}-lem}, we get that when $g'$ is in form of \eqref{monom-g} with respect to the lifting $\sigma'$,
\begin{align}\label{sum-s1m-sigma'}
S^{(g')}_{1|m}(\mathcal X,\mathcal L)
=&{{\mathbf{k}}}!\frac{(\mathcal L^{\sigma'[{\mathbf k}]})^{n+|{\mathbf k}|+1}}{(n+|{\mathbf k}|+1)!}m^{n+1}-{{\mathbf{k}}}!\frac{K_{\mathcal X/\mathbb P^1}^{[{\mathbf k}]}\cdot(\mathcal L^{\sigma'[{\mathbf k}]})^{n+|{\mathbf k}|}}{2(n+|{\mathbf k}|)!}m^{n}
\notag\\&+O(m^{n-1}),~m\to+\infty.
\end{align}

Now suppose that $g$ is a monomial \eqref{monom-g} on the canonical polytope $\Delta$. Then it corresponds to
$$g'(\eta)=g(\eta-\mu)=\sum_{\mathbf0\leq\mathbf i\leq\mathbf k}C_{\mathbf k!}^{\mathbf i!}(-1)^{|\mathbf i|}\prod_{A=1}^r\mu^{\mathbf i}\eta^{\mathbf{k-i}}$$
on $\Delta'$. Apply \eqref{sum-s1m-sigma'} to each $\eta^{\mathbf{k-i}}$ and using linearity, we see that when $g$ is given by \eqref{monom-g} with respect to the canonical lifting $\sigma_0$,
\begin{align}\label{sum-s1m-pol}
S^{(g)}_{1|m}(\mathcal X,\mathcal L)
=&{{\mathbf{k}}}!\frac{(\mathcal L^{\sigma_0[{\mathbf k}]})^{n+|{\mathbf k}|+1}}{(n+|{\mathbf k}|+1)!}m^{n+1}-{{\mathbf{k}}}!\frac{K_{\mathcal X/\mathbb P^1}^{[{\mathbf k}]}\cdot(\mathcal L^{\sigma_0[{\mathbf k}]})^{n+|{\mathbf k}|}}{2(n+|{\mathbf k}|)!}m^{n}
\notag\\&+O(m^{n-1}),~m\to+\infty.
\end{align}

Also, it is direct to check
\begin{align}\label{sum-s2m-pol}
S^{(g)}_{2|m}(\mathcal X,\mathcal L)=&\frac12\sum_{(\eta;\lambda)\in\mathfrak X(T)\oplus\mathbb Z}\frac{\lambda}{m} \sum_{A=1}^r\frac{\eta_A}{m}\frac{\partial g}{\partial \eta_A}(\frac{\eta_A}{m})\dim{\rm H}^0(\mathcal X_0,m\mathcal L_0)^{(T\times T_0)}_{(\eta;\lambda)}\notag\\
=&\frac{|\mathbf k|}2{{\mathbf{k}}}!\frac{(\mathcal L^{\sigma_0[{\mathbf k}]})^{n+|{\mathbf k}|+1}}{(n+|{\mathbf k}|+1)!}m^{n}+O(m^{n-1}),~m\to+\infty.
\end{align}
Combining \eqref{sum-s1m-pol}-\eqref{sum-s2m-pol} with the fact that
\begin{align}\label{sum-h0}
{\rm h}^0(X,-mK_X)=V\frac{m^n}{n!}(1+\frac12nm^{-1}+O(m^{-2})),~m\to+\infty
\end{align}
for $L=-K_X$, we see that \eqref{S1-S2} holds, and the $g$-modified Futaki invariant is given by \eqref{Fut-g-monomial}.

Now we prove \eqref{Fut-general-g-def}. Suppose that $g$ polynomial satisfying \eqref{g-polynomial}. From the linearality of \eqref{g-polynomial}, it is easy to check that \eqref{sum-s1m-pol}-\eqref{sum-s2m-pol} also holds. By a direct computation one gets \eqref{Fut-general-g-def}.

Recall the fact that when $g$ is a monomial, $\sum_{A=1}^{r}\eta_A\frac{\partial g}{\partial \eta_A}=k_Ag$. 
By linearality of \eqref{Fut-general-g-def}, we directly see \eqref{S1-S2} and holds for any polynomial $g$.


\end{proof}

Recall \eqref{K-log-def}, we have $K_{\mathcal X/\mathbb P^1}=K_{\mathcal X/\mathbb P^1}^{\log}+(\mathcal X_0-\mathcal X_{0,{\rm red}})$. 
By Theorem \ref{Fut-g-smooth-thm} and the definition of ${\rm M}_g^{\rm NA}(\cdot)$, we have:
\begin{prop}\label{Fut=F-g}
Let $X$ be a $\mathbb Q$-Fano variety with effective $T$-action and $g\geq0$ a polynomial on $\Delta$ given by \eqref{g-polynomial}. Suppose that $(\mathcal X,\mathcal L)$ is a $T$-equivariant normal test configuration. Then $\frac V{V_g}{\rm Fut}_g(\mathcal X,\mathcal L)\geq{\rm M}_g^{\rm NA}(\mathcal X,\mathcal L)$. Moreover, the equality holds if and only if $(\mathcal X,\mathcal L)$ has reduced central fibre.
\end{prop}

\begin{proof}
Suppose that $\mathcal X_{0,{\rm red}}=\sum_{a=1}^{n_0}\mathcal X_{0,a}$ and $\mathcal X_0=\sum_{a=1}^{n_0}\mathcal X_{0,a}$ for reduced, irreducible varieties $X_{0,1},...,X_{0,n_0}$ and positive integers $m_1,...,m_{n_0}$. It suffices to show that under certain lifting $\sigma$ of the $T$-action,
\begin{align}\label{intersect-positive}
&\sum_{\mathbf k}a_{\mathbf k}(\mathcal L^{\sigma[\mathbf k]})^{n+|\mathbf k|+1}(\mathcal X_0-\mathcal X_{0,{\rm red}})^{[\mathbf k]}\notag\\
=&\sum_{a=1}^{n_0}(m_a-1)\sum_{\mathbf k}a_{\mathbf k} (\mathcal L^{\sigma[\mathbf k]})^{n+|\mathbf k|+1}\mathcal X_{0,a}\geq0.
\end{align}
In fact, from the change of lifting formulas \eqref{change-of-g-def} and \eqref{change-of-L}, it is direct to check that \eqref{intersect-positive} does not depends on the choice of lifting of the $T$-action on $-K_X$.

On the other hand, since $T$ acts on each $(\mathcal X_{0,a},\mathcal L|_{\mathcal X_{0,a}})$, by \cite[Section 2]{Han-Li} (see also Lemma \ref{intersection-integral} in the Appendix),
\begin{align*}
\sum_{\mathbf k}a_{\mathbf k}(\mathcal L^{\sigma[\mathbf k]})^{n+|\mathbf k|+1}\mathcal X_{0,a}=&\sum_{\mathbf k}a_{\mathbf k} (\mathcal L|_{\mathcal X_{0,a}}^{\sigma[\mathbf k]})^{n+|\mathbf k|}\\
=&\frac{(n+|\mathbf k|)!}{\mathbf k!}\int_{\Delta_+(\mathcal L|_{\mathcal X_{0,a}})}g(\lambda){\mathbf{DH}}_T(\mathcal X_{0,a},\mathcal L|_{\mathcal X_{0,a}})(\lambda),
\end{align*}
and we get \eqref{intersect-positive}.
\end{proof}

\begin{rem}
By using the equivariant Riemann-Roch formula, it is showed in \cite[Section 1]{WZZ} that the modified Futaki invariant exists when $g$ is exponential of the potential of the soliton vector field.
\end{rem}

\subsection{Variants of Stability}
\subsubsection{K-polistability} In the sense of \cite{Do,WZZ} we have
\begin{defi}\label{K-stab-def-Fut}
We say that a $\mathbb Q$-Fano $\mathfrak G$-variety $X$ is ($\mathfrak G$-equivariantly) $g$-modified K-semistable if the $g$-modified Futaki invariant for any $\mathfrak G\times T$-equivariant test configuration  is nonnegative, and is ($\mathfrak G$-equivariantly) $g$-modified K-polystable if in addition the $g$-modified Futaki invariant vanishes precisely on product $\mathfrak G\times T$-equivariant test configurations. When $X$ is not ($\mathfrak G$-equivariantly) $g$-modified K-semistable, we say it is $g$-modified K-unstable.
\end{defi}
Also, in the sense of \cite{Han-Li-KRS}, one can define the $g$-K-stability using the $g$-weighted non-Archimedean Mabuchi functional:
\begin{defi}\label{K-stab-def-MNA}
We say that a $\mathbb Q$-Fano $\mathfrak G$-variety $X$ is ($\mathfrak G$-equivariantly) $g$-K-semistable if the $g$-weighed non-Archimedean Mabuchi functional for any $\mathfrak G\times T$-equivariant test configuration  is nonnegative, and is ($\mathfrak G$-equivariantly) $g$-K-polystable if in addition the $g$-weighed non-Archimedean Mabuchi functional vanishes precisely on product $\mathfrak G\times T$-equivariant test configurations. When $X$ is not ($\mathfrak G$-equivariantly) $g$-K-semistable, we say it is $g$-K-unstable.
\end{defi}

By Proposition \ref{Fut=F-g}, we can prove that the $\mathfrak G$-equivariantly $g$-modified K-polystability and $\mathfrak G$-equivariantly $g$-K-polystability coincide with each other, provided $g$ is a polynomial with non-negative coefficients.
\begin{prop}\label{K-stab-equiv}
Let $X$ be a $\mathbb Q$-Fano variety with a reductive group $\mathfrak G$-action. Let $g\geq0$ be a polynomial. Then to test the $\mathfrak G$-equivariantly $g$-modified K-polystability or $g$-K-polystability of $X$, it suffices to consider $\mathfrak G$-equivariant normal test configurations with reduced central fibre. Consequently, $X$ is $\mathfrak G$-equivariantly $g$-modified K-polystable in the sense of Definition \ref{K-stab-def-Fut} if and only if it is $\mathfrak G$-equivariantly $g$-equivariantly $g$-K-polystable in the sense of Definition \ref{K-stab-def-MNA}.
\end{prop}
\begin{proof}
Suppose that there is a non-product $\mathfrak G$-equivariant normal test configuration $(\mathcal X,\mathcal L)$ so that ${\rm Fut}_g(\mathcal X,\mathcal L)\leq0$. Then by Proposition \ref{Fut=F-g}, ${\rm M}^{\rm NA}_g(\mathcal X,\mathcal L)\leq0$. By \cite[Section 5.1]{LL-202206}, there is a sufficiently divisible $d\in\mathbb N_+$ so that the test configuration $(\mathcal X',\mathcal L')$ is non-product and has reduced central fibre. Then by Proposition \ref{Fut=F-g},
\begin{align}\label{K-stab-def-compare}
{\rm Fut}_g(\mathcal X',\mathcal L')={\rm M}^{\rm NA}_g(\mathcal X',\mathcal L')=d{\rm M}^{\rm NA}_g(\mathcal X,\mathcal L)\leq0.
\end{align}

On the other hand, suppose that there is a non-product $\mathfrak G$-equivariant normal test configuration $(\mathcal X,\mathcal L)$ so that ${\rm M}^{\rm NA}_g(\mathcal X,\mathcal L)\leq0$. Take $(\mathcal X',\mathcal L')$ as above, it holds \eqref{K-stab-def-compare}. Hence we get the Proposition.

\end{proof}

We will see that Proposition \eqref{K-stab-equiv} holds for arbitrary continuous $g$ when $X$ is a $\mathbb Q$-Fano spherical variety. See Section 5 below.

\subsubsection{$\mathfrak G$-uniform $g$-stability}
The following uniform stability are closely related to the existence of KR $g$-soliton:
\begin{defi}
Let $X$ be a $\mathbb Q$-Fano variety with a complex torus $T$-action. Let $\mathfrak G$ be a connected reductive subgroup of ${\rm Aut}_{T} (X)$ and $\mathfrak T$ be its
centre. Then we say that $X$ is $\mathfrak G$-uniformly $g$-Ding-stable ($g$-K-stable, resp.) if there exists a constant $\epsilon_0>0$ such that for any $\mathfrak G\times T$-equivariant test configuration $(\mathcal X,\mathcal L)$ it holds
$${\rm D}_g^{\rm NA}(\mathcal X,\mathcal L)~ ({\rm M}_g^{\rm NA}(\mathcal X,\mathcal L)~\text{resp.}) \geq\epsilon_0 \cdot \inf_{\sigma\in\mathfrak T}{\rm J}^{\rm NA}_g(\sigma^*(\mathcal X,\mathcal L)).$$
Here by $\sigma^*(\mathcal X,\mathcal L)$ we mean the twist of $(\mathcal X,\mathcal L)$ by $\sigma$.
\end{defi}
The precise relation ship between $\mathfrak G$-uniformly $g$-Ding-stability and existence of KR $g$-solitons are studied in \cite[Section 6]{Han-Li-KRS}.

\section{Polarized $G$-spherical varieties}
\subsection{Preliminaries on spherical varieties}
In the following we overview the theory of spherical varieties. The origin goes back to \cite{Luna-Vust}. We use \cite{Timashev-survey, Timashev-book} as main references.
\begin{defi}
Let $G$ be a connected, complex reductive group. A closed subgroup $H\subset G$ is called a {spherical subgroup of $G$} if there is a Borel subgroup $B$ of $G$ acts on $G/ H$ with an open orbit.
In this case $G/H$ is called a {spherical homogeneous space}. A {spherical embedding} of $G/H$ (or simply a spherical variety) is a normal variety $X$ equipped with a $G$-action so that there is an open dense $G$-orbit isomorphic to $G/H$.
\end{defi}

\subsubsection{Homogenous spherical datum}
Let $H$ be a spherical subgroup of $G$ with respect to the Borel subgroup $B$. The action of $G$ on the function field $\mathbb C(G/H)$ of $G/ H$ is given by
$$(g^*f)(x):=f(g^{-1}\cdot x),~\forall g\in G, x\in G/H~\text{and}~f\in\mathbb C(G/H).$$
A function $f(\not=0)\in\mathbb C(G/H)$ is called \emph{$B$-semiinvariant} if there is a character of $B$, denote by $\varpi$ so that $b^*f=\varpi(b)f$ for any $b\in B$.
By \cite[Section 25.1]{Timashev-book}, $\mathbb C(G/H)^B=\mathbb C$. Two $B$-semiinvariant functions associated to a same character can differ from each other only by multiplying a non-zero constant.

Let $\mathfrak M(G/H)$ be the lattice of $  B$-characters which admits a corresponding $B$-semi-invariant functions, and $\mathfrak N(  G/  H)={\rm Hom}_\mathbb Z(\mathfrak M(  G/  H),\mathbb Z)$ its $\mathbb Z$-dual. The rank $r_0$ of $\mathfrak M(G/H)$ is called the \emph{rank of $G/H$}. There is a map $\varrho$ which maps a valuation $D$ of $\mathbb C(  G/  H)$ to an element $\varrho_D$ in $\mathfrak N_\mathbb Q(G/H)=\mathfrak N(G/H)\otimes_\mathbb Z\mathbb Q$ by
$$\varrho_D(\varpi)=\nu(f),$$
where $f\in\mathbb C(G/H)^{(B)}_\varpi$. Again, this is well-defined since $\mathbb C(G/H)^B=\mathbb C$. It is a fundamental result that $\varrho$ is injective on $  G$-invariant valuations and the image forms a convex cone $\mathcal V(  G/  H)$ in $\mathfrak N_\mathbb  Q(  G/  H)$, called the \emph{valuation cone} of $  G/  H$ (cf. \cite[Section 19]{Timashev-book}).
 Moreover, $\mathcal V(  G/  H)$ is a solid cosimplicial cone 
which is a (closed) fundamental chamber of a certain crystallographic reflection group, called the {\it little Weyl group} (denoted by $W^{G/H}$, cf. \cite[Sections 22]{Timashev-book}). In fact, $W^{G/H}$ is the Weyl group of the \emph{spherical root system} $\Phi^{G/H}$ of $  G/  H$ (cf. \cite[Section 30]{Timashev-book}).
The simple spherical roots are defined to be the primitive generators of edges of $(-\mathcal V(G/H))^\vee$. The set of simple spherical roots is denoted by $\Pi_{G/H}$.

A $  B$-stable prime divisors in $  G/  H$ is called a \emph{colour}. Denote by $\mathcal D(  G/  H)$ the set of colours. A colour $D\in\mathcal D(  G/  H)$ also defines a valuation on $  G/  H$. However, the restriction of $\varrho$ on $\mathcal D(  G/  H)$ is usually non-injective.

Now we briefly introduce the homogeneous spherical datum, which by a deep result in \cite{Losev} (cf. \cite[Theorem 30.22]{Timashev-book}) characterizes the spherical homogeneous space. Let $P_\alpha$ be the minimal standard parabolic subgroup of $G$ containing $B$ corresponding to the simple root $\alpha\in\Pi_G$. Set
$$\mathcal D(\alpha):=\{D\in\mathcal D(G/H)|D~\text{is not}~P_\alpha\text{-stable}\}.$$
Then $\mathcal D(G/H)=\cup_{\alpha\in\Pi_G} \mathcal D(\alpha)$. We see that a colour  $D\in\mathcal D(G/H)$ is of
\begin{itemize}
\item type a (denote the collection by $\mathcal D^a(G/H)$): if $D\in\mathcal D(G/H)$ for $\alpha\in\Phi^{G/H}$;
\item type a' (denoted by $\mathcal D^{a'}(G/H)$): if $D\in\mathcal D(G/H)$ for $2\alpha\in\Phi^{G/H}$;
\item type b (denoted by $\mathcal D^b(G/H)$): Otherwise.
\end{itemize}
Note that although a colour $D$ may belong to different $\mathcal D(\alpha)$, the type of $D$ is well-defined. Also, set
$$\Pi_{G/H}^p:=\{\alpha\in\Pi_G|\mathcal D(\alpha)=\emptyset\},$$
and $\mathcal D^a(G/H)$ the set of all colours of type a.
\begin{defi}
The quadruple $(\mathfrak M(G/H), \Phi^{G/H}, \Pi_{G/H}^p, \mathcal D^a(G/H))$ is called the homogeneous spherical datum of $G/H$.
\end{defi}
The homogeneous spherical datum was introduced by \cite{Luna}. It is proved by \cite{Losev} that the homogeneous spherical datum uniquely
determines $G/H$ up to $G$-equivariant isomorphism. The axioms that an abstract quadruple $(\bar{\mathfrak M}, \bar\Phi, \bar\Pi^p, \bar{\mathcal D}^a)$ forms a homogeneous spherical datum can be found in \cite[Section 30]{Timashev-book}.

\subsubsection{Line bundles and polytopes}
Let $X$ be a complete spherical variety, which is a spherical embedding of some $  G/  H$.
Let $L$ be a $  G$-linearlized line bundle on $X$.
In the following we will associated to $(X,L)$ several polytopes, which encode the geometric structure of $X$.
\subsubsection*{Moment polytope of a line bundle}
Let $(X,L)$ be a polarized spherical variety. Then for any $k\in\mathbb N$ we can decompose ${\rm H}^0(X,L^k)$ as direct sum of irreducible $G$-representations,
\begin{align}\label{H0kL}
{\rm H}^0(X,L^k)=\bigoplus_{ \lambda\in \Delta_{L,k} } V_{ \lambda},
\end{align}
where $\Delta_{L,k}$ is a finite set of $  B$-weights and each $  V_{ \lambda}$ is called an \emph{isotypic component}, which is isomorphic to the irreducible representation of $G$ with highest weight $ \lambda$.
Set
$$\Delta_+(L):=\overline{\cup_{k\in\mathbb N}(\frac1k\Delta_{L,k})}.$$
.
Denote by $\mathfrak X(B)$ the lattice of $B$-weights. Then $\Delta_+(L)$ is indeed a polytope in $\mathfrak X_{\mathbb R}(B)$. Moreover, denote by $\Phi_+^G$ the positive roots of $G$ with respective to $B$, then $\Delta_+(L)$ lies in the dominant Weyl chamber determined by $\Phi_+^G$. We call $\Delta_+(L)$ the \emph{moment polytope of $(X,L)$} (cf. \cite[Section 17]{Timashev-book}). Clearly, the moment polytope of $(X,L^k)$ is $k$-times the moment polytope of $(X,L)$ for any $k\in\mathbb N_+$.

We also introduce here a useful weight function $\pi(\lambda)$ defined for $\lambda\in\Delta_+(L)$. Suppose that $\lambda\in\mathfrak M(G/H)$. By Weyl character formula \cite[Section 3.4.4]{Zhelobenko-Shtern},
\begin{align*}\dim(V_\lambda)=&\frac{\prod_{\alpha\in\Phi_+^G,\alpha\not\perp\Delta_+(L)}\langle\alpha,\rho+k\lambda\rangle}{\prod_{\alpha\in\Phi_+^G,\alpha\not\perp\Delta_L}\langle\alpha,\rho\rangle},\\
=&C_{G/H}(\pi(\lambda)+\rho(\nabla\pi(\lambda))+(\text{lower order terms})),
\end{align*}
where $\rho:=\frac12\sum_{\alpha\in\Phi_+^G}\alpha$, the constant $$C_{G/H}=\frac1{\prod_{\alpha\in\Phi_+^G,\alpha\not\perp\Delta_+(L)\langle\alpha,\rho\rangle}},$$
and
\begin{align}\label{pi-weight}\pi(\lambda)=\prod_{\alpha\in\Phi_+^G,\alpha\not\perp\Delta_+(L)}\langle\alpha,\lambda\rangle.
\end{align}

\subsubsection*{Polytope of a divisor}
Recall that the spherical embedding $X$ of $G/H$ is uniquely determined by its coloured fan $\mathfrak F_X$ (cf. \cite{Luna-Vust, Timashev-survey, Timashev-book}). Denote by $\mathcal I_{  G}(X)=\{D_1,...,D_{d_0}\}$ the set of $  G$-invariant prime divisors in $X$.
Then any $D\in\mathcal I_{  G}(X)$ corresponds to a $1$-dimensional cone $(\mathfrak C_D,\emptyset)$ in the coloured fan $\mathfrak F_X$ of $X$.
Denote by $u_D$ the prime generator of $\mathfrak C_D$. Recall that $\mathcal D(  G/  H)$ is the set of colours, which are $  B$-stable but not $  G$-stable in $X$. Any $  B$-stable $\mathbb Q$-Weil divisor can be written as
\begin{align}\label{weil-div}
{\mathfrak d}=\sum_{D\in\mathcal I_G(X)}c_DD+\sum_{D\in\mathcal D(  G/  H)}c_DD
\end{align}
for some $c_D\in\mathbb Q$. Set
$$\mathcal D_X:=\cup\{\mathfrak R\subset\mathcal D(  G/  H)|~\exists(\mathfrak C,\mathfrak R)\in\mathfrak F_X\}.$$
By \cite[Proposition 3.1]{Brion89} (cf. \cite[Section 17]{Timashev-book}), ${\mathfrak d}$ is further a $\mathbb Q$-Cartier divisor if and only if there is a rational piecewise linear function $l_{\mathfrak d}(\cdot)$ on $\mathfrak F_X$ such that
$$c_D=l_{\mathfrak d}(u_D),~D\in\mathcal I_G(X)~\text{and}~c_D=l_{\mathfrak d}(\varrho_D),~\forall D\in\mathcal D_X.$$
It is further proved in that when ${\mathfrak d}$ is ample $l_{\mathfrak d}(-x):\mathfrak N_\mathbb R(  G/  H)\to\mathbb R$ is the support function of some convex polytope $\Delta_{\mathfrak d}\subset\mathfrak M_\mathbb R(G/H)$. We call the $\Delta_{\mathfrak d}$ the \emph{polytope of ${\mathfrak d}$}.

Suppose that $s$ is a $  B$-semiinvariant section of $L$ with respect to a character $\chi$. Let ${\mathfrak d}$ be the divisor of $s$. We have
\begin{prop}\label{polytope-prop}(cf. \cite[Theorem 3.30]{Timashev-survey})
The two polytopes $\Delta_L$ and $\Delta_{\mathfrak d}$ are related by $$\Delta_L=\chi+\Delta_{\mathfrak d}.$$
\end{prop}

Thus, fix any $B$-semiinvariant rational section $s_0$ of $L$ with $B$-weight $\chi_0$, then $\Delta_+(L)\subset\mathfrak A_\mathbb R(L):=\mathfrak M_\mathbb R(G/H)+\chi_0$. Indeed, $\mathfrak A_\mathbb R(L)$ does not depends on the choice of $s_0$ and is the affine span of $\Delta_+(L)$. Note that when a finite covering of $G/H$ is $\mathbb Q$-factorial, or $G/H$ is the flat limit of some space with $\mathbb Q$-factorial finite covering, we can choose $\chi_0\in\mathfrak M_\mathbb R$ and therefore $\Delta_L\subset\mathfrak M_\mathbb R(G/H)$.

\subsubsection{The anticanonical line bundle} When $X$ is ($\mathbb Q$-)Fano and $L=-K_X$, \cite{Ga-Ho} proved there is a canonical choice of $B$-invariant divisor ${\mathfrak d_0}$ of $-K_X$ in form of \eqref{weil-div},
\begin{align*}
{\mathfrak d_0}=\sum_{A=1}^{d_0}D_A+\sum_{D\in\mathcal D_{  B}(  G/  H)}m_DD,
\end{align*}
where the coefficients $m_D$'s are explicitly obtained in \cite{Ga-Ho} according to the type of each colour $D$ (cf. \cite[Theorem 1.5]{Ga-Ho}). In fact, this divisor corresponds to a $B$-semiinvariant section $\mathfrak d_0$ of
$-K_X$ (in case $X$ is Gorenstein Fano) with $B$-weight
\begin{align}\label{kappa}
\kappa_P=\sum_{\alpha\in\Phi^G_+,\alpha\not\perp\Delta_+(-K_X)}\alpha.
\end{align}
The polytope of $\mathfrak d_0$ is the ($\mathbb Q$-)reflexible polytopes defined in \cite{Ga-Ho}.

\subsubsection{The $G$-equivariant automorphism group}
Let $X$ be a $G$-spherical variety, which is a spherical embedding of some $G/H$. It is known that the $G$-automorphism group ${\rm Aut}_{G}(G/H)=N_G(H)/H$ and is a commutative group (cf. \cite[Proposition 1.8]{Timashev-book}). The action of $G\times N_G(H)/H$ on $G/H$ is defined as
$$(g,p)g_0H:=gg_0p^{-1}H,~\forall g,g_0\in G~\text{and}~p\in N_G(H)/H.$$
This action is well-defined since $p^{-1}Hp=H$. Its neutral component ${\rm Aut}^0_G(G/H)$ is isomorphic to the neutral compone ${\rm Aut}^0_G(X)$ of ${\rm Aut}_G(X)$ (cf. \cite{Losev}).

It is known that for a spherical variety $X$, ${\rm Aut}^0_G(X)$ is a complex torus (cf. \cite[Corollary 6.5]{Knop-JAMS}). Let $(X,L)$ be a polarized spherical variety, then the Kodaira ring (homogeneous coordinate ring) of $X$ is
\begin{align*}
R(X,L)=\oplus_{k\in\mathbb N}R_k,~R_k={\rm H}^0(X,kL)=\oplus_{\lambda\in\Delta_{L,k}}V_\lambda.
\end{align*}
The group ${\rm Aut}^0_G(X)$ acts on $R(X,L)$ preserving each $R_k$. Let $\xi\in\mathfrak{aut}^0_G(X)$ be a rational element which generates a $1$-dimensional torus $T_\xi$-action. Then each $R_k$ can be decomposed into direct sums of irreducible $T_\xi$-representations. Since the $T_\xi$-action commutes with the $G$-action, each isotypic component $V_\lambda,~\lambda\in\Delta_{L,k}$ is $T_\xi$-invariant, and $\xi$ acts on any $s\in V_\lambda$ through a common weight $\nu_\xi(\lambda)$.

Suppose that $\lambda_i\in\Delta_{L,k_i},i=1,2$. Then for any $s_i\in V_{\lambda_i}$, it holds
\begin{align*}
\xi(s_1\cdot s_2)=\xi(s_1)+\xi(s_2).
\end{align*}
On the other hand, by \cite[Proposition 3.1]{LL-202206},
\begin{align}\label{mult-rule}
V_{\lambda_1}\cdot V_{\lambda_2}=\bigoplus_{\lambda_1+\lambda_2-\beta_i}V_{\lambda_1+\lambda_2-\beta_i},
\end{align}
where each $\beta_i$ is a non-negative $\mathbb Q$-linear combination of simple spherical roots. In particular there is a component with $\beta=0$. Thus for each $\beta_i$ appeared in \eqref{mult-rule},
\begin{align}\label{nu-xi-(s1s2)}
\nu_\xi(\lambda_1+\lambda_2-\beta_i)=\nu_\xi(\lambda_1+\lambda_2)=\xi(\lambda_1)+\xi(\lambda_2).
\end{align}
Combining with \cite[Remark 3.3]{LL-202206} we can conclude from the first equality of \eqref{nu-xi-(s1s2)} that
\begin{align}\label{ortho-to-sph-roots}
\nu_\xi(\alpha)=0,~\alpha\in\Phi^{G/H}.
\end{align}
On the other hand, from the second equality of \eqref{nu-xi-(s1s2)}, it holds
$$\nu_\xi(p\lambda)=p\nu_\xi(\lambda),~\forall p\in\mathbb N,$$
and
$$\nu_\xi(\frac{k_1}{k_1+k_2}\frac{\lambda_1}{k_1}+\frac{k_2}{k_1+k_2}\frac{\lambda_2}{k_2})=\frac{k_1}{k_1+k_2}\nu_\xi(\frac{\lambda_1}{k_1})+\frac{k_2}{k_1+k_2}\nu_\xi(\frac{\lambda_2}{k_2}).$$
Thus $\nu_\xi(\cdot)$ descends to an affine function
\begin{align*}
\nu_\xi(\lambda)=V_\xi(\lambda)+\chi_\xi,~\lambda\in\Delta_+(L)
\end{align*}
on $\Delta_+(L)$ so that for each $\lambda\in\Delta_{L,k}$,
\begin{align}\label{nu-xi-Delta}
\nu_\xi(\lambda)=kV_\xi(\frac1k\lambda)+k\chi_\xi.
\end{align}
Moreover, $V_\xi\in\mathcal V_z(G/H):=\mathcal V(G/H)\cap(-\mathcal V(G/H))$, the central part of $\mathcal V(G/H)$, and different choices of the constant $\chi_\xi$ correspond to different liftings of the $T_\xi$-action on $L$. In the following, we will identify $\xi\in\mathfrak{aut}^0_G(X)$ with $V_\xi\in\mathcal V_z(G/H)$.

Suppose that $T\subset{\rm Aut}^0_G(X)$ is an $r$-dimensional torus. Then we can choose a set of generators $\{\xi_A\}_{A=1}^r\subset\mathcal V_z(G/H)$. Let $\xi_A^*\in\mathfrak t^*$ be the dual of $\xi_A$. The $T$-weights on each $R_k$ is given by \eqref{nu-xi-Delta} and each choice of the character $\chi=\sum_{A=1}^r\chi_A\xi_A^*$ of $T$ correspond to a lifting of the $T$-action on $L$.

As showed in \cite[Theorem 4.2]{Li-Wang}, the character associated to the canonical lifting is
\begin{align}\label{chi-0-can}
\chi_0=-\sum_{A=1}^r{\kappa_P}_A\xi_A^*,
\end{align}
the restriction of $-\kappa_P$ on $\mathfrak t$.

\subsection{The polarized variety $(X^{[\mathbf{k}]},L^{[\mathbf{k}]})$}

In this section we compute the combinatorial data of $(X^{[\mathbf{k}]},L^{[\mathbf{k}]})$ for a general polarized $G$-spherical variety $(X,L)$. Let $(X,L)$ be a polarized spherical embedding of $G/H$. We further assume that $L$ is $G\times N_G(H)$-linearized.

Suppose that $\xi\in\mathfrak{aut}_G(X)$ which generates a rank $r$ torus $T\subset{\rm Aut}_G^0(X)$. Denote by $\mathfrak t$ the Lie algebra of $T$. Fix a lifting of the $T$-action on $L$ with corresponding $T$-character $\chi$. 
From the previous section we have the embedding
$${\mathfrak t}\overset{\iota_1}{\hookrightarrow}\mathfrak{aut}_G(X)\cong\mathcal V_{z\mathbb R}(G/H)\overset{\iota_2}{\hookrightarrow}\mathfrak g,$$
and $(e,p)H=p^{-1}H=(p^{-1},e)H$ for any $p\in N_G(H)/H$. Then $\mathfrak t$ is identified with an $r$-dimensional rational linear subspace of $\mathcal V_z(G/H)$, and the moment map $\mathbf{m}_{\omega_\phi}:X\to\Delta\in\mathfrak t^*$ can be decomposed as
$\mathbf{m}_{\omega_\phi}(\cdot)={\mathfrak r}\circ\mathbf{\mu}_{\omega_\phi}(\cdot)+\chi$, where $$\mathbf{\mu}_{\omega_\phi}:X\to\Delta_+(L)\subset\mathfrak X_\mathbb R(B)$$
is the moment map with respect to the $T_0$-action for the maximal torus $T_0=B\cap B^-$ of $G$, and $\mathfrak r:\mathfrak t^*_0\to\mathfrak t^*$ is the restriction map defined by $\mathfrak r(\lambda)=\lambda|_\mathfrak t,~\forall\lambda\in\mathfrak X_\mathbb R(B)(=\mathfrak t^*_0)$. Here we consider the restriction $\mathfrak r(\lambda)$ of an element $\lambda\in\mathfrak t_0^*$ on $\mathfrak t$ as an element in $\mathfrak t^*$. Consequently,
$$\Delta=\{\mathfrak r(\lambda)+\chi\in\mathfrak t^*|\lambda\in\Delta_+(L)\}.$$



Recall the construction of $(X^{[\mathbf{k}]},L^{[\mathbf{k}]})$ in Section 2.1, \emph{Step-1}.
Let $\{\xi'_A\}_{A=1}^r$ be a set of generators of $\mathbb T$ and $\xi_A=\iota(\xi'_A)$. Then $\xi_A\in\mathfrak t,A=1,...,r$ generate $T$. Also denote by $T^{[\mathbf{k}]+\mathbf{1}}=\prod_{A=1}^r{(\mathbb C^*)}^{k_A+1}$ the $(|\mathbf k|+r)$-dimensional complex torus. We have:
\begin{prop}\label{Xk-Lk-structure}
Let $(X,L)$ be a polarized spherical embedding of $G/H$ with moment polytope $\Delta_+(L)$. Suppose that the lifting of $T$ on $L$ is chosen so that the moment polytope $\Delta$ lies in the first quadrant of $\mathfrak t^*$, that is, $\xi_A(\lambda+\chi)\geq0$ for any $A\in\{1,...,r\}$ and $\lambda\in\Delta$. Then $(X^{[\mathbf{k}]},L^{[\mathbf{k}]})$ is a polarized $G\times T^{[\mathbf{k}]+\mathbf{1}}$-spherical variety with polytope
\begin{align}\label{Xk-Lk-polytope}
\Delta_+(L^{[\mathbf k]})=\{(\lambda;\mu_{(A),i_A})\in\Delta_+(L)\oplus\oplus_{A=1}^r\mathbb R^{k_A+1}_{\geq0}|~\lambda_A+\chi_A=\sum_{i_A=0}^{k_A}\mu_{(A),i_A}\},
\end{align}
where $\chi$ is the $T$-character corresponding to the lifting and $\lambda_A$ is the restriction of the character $\lambda$ on the $1$-dimensional torus $\exp(t\xi_A)$.
\end{prop}

\begin{proof}
For each $A\in\{1,...,r\}$ choose an $i_A\in\{0,...,k_A\}$. Set $[i]=(i_1,...,i_r)$ and $\mathcal U_{[i]}=\mathbb C^{[\mathbf{k}]+\mathbf{1}}\cap\{z^{(A),i_A}\not=0|A=1,...,r\}$. It is easy to see that the open set
$$X^{[\mathbf{k}]}_{[i]}:=\{[(x;z^{(B),i_B})]\in X^{[\mathbf{k}]}|z^{(A),i_A}\not=0,~A=1,...,r\}\cong X\times\mathcal U_{[i]}.$$
We may choose local coordinates $x,\zeta^{(B),i_B}_{[i]}:=\frac{z^{(B),j_B}}{z^{(B),i_B}}$ on $X^{[\mathbf{k}]}_{[i]}$ for $B=1,...,r$ and $j_B=0,...,k_B$ with each $j_B\not=i_B$. Suppose that $[(x;z^{(B),i_B})]\in X^{[\mathbf{k}]}_{[i]}\cap X^{[\mathbf{k}]}_{[i']}$. Then
\begin{align*}
&[(x,\zeta^{(1),0}_{[(A),i_A]},...,\zeta^{(1),k_1}_{[i]},...,
\underbrace{{\zeta^{(B),1}_{[i]},...,\zeta^{(B),i_B-1}_{[i]},1,\zeta^{(B),i_B+1}_{[i]},...,\zeta^{(A),k_A}_{[i]}}}_{\text{the $B$-th group}},\\&...,
\zeta^{(r),0}_{[i]},...,\zeta^{(r),k_r}_{[i]})]\\
=&[(\iota(\frac{z^{(1),i'_{1'}}}{z^{(A),i_A}},...,\frac{z^{(r),i'_{r'}}}{z^{(r),i_r}})x,\zeta^{(1),0}_{[i']},...,\zeta^{(r),k_r}_{[i']})].
\end{align*}
Clearly the collection $\{X^{[\mathbf{k}]}_{[i]}\}_{[i]}$ gives an open covering of $X^{[\mathbf{k}]}$. Consequently $X^{[\mathbf{k}]}$ is normal.

Then we prove that $X^{[\mathbf{k}]}$ is spherical. Consider the $G\times T^{[\mathbf{k}]+\mathbf{1}}$-action on $X^{[\mathbf{k}]}$. It suffices to prove the stabilizer of $[(eH,1,...,1)]$ is spherical in $G\times T^{[\mathbf{k}]+\mathbf{1}}$.
Suppose that $(g;t^{(A),i_A})\in G\times T^{[\mathbf{k}]+\mathbf{1}}$ so that
$$(g;t^{(A),i_A})[eH,1,...,1]=[eH,1,...,1].$$
Then
\begin{align*}
\left\{\begin{aligned} &gH=\exp(-\sum_Ac_A\xi_A)H,\\&t^{(A),i_A}=e^{c_A},\end{aligned}\right.
\end{align*}
where $c_1,...,c_r$ are $r$ constants in $\mathbb C$. Thus we see that
$${\rm Stab}_{G\times T^{[\mathbf{k}]+\mathbf{1}}}([eH,1,...,1])=\{(\iota(\vartheta^{-1})h,\vartheta)|~h\in H,\vartheta\in\mathbb T\}.$$
Since $H$ is a spherical in $G$, we see that ${\rm Stab}_{G\times T^{[\mathbf{k}]+\mathbf{1}}}([eH,1,...,1])$ is spherical in ${G\times T^{[\mathbf{k}]+\mathbf{1}}}$.

We are going to compute the combinatory data of $X^{[\mathbf{k}]}$. Write $\hat G=G\times T^{[\mathbf{k}]+\mathbf{1}}$ and $\hat B=B\times T^{[\mathbf{k}]+\mathbf{1}}$ for short. It is direct to determine $\mathbb C(X^{[\mathbf{k}]})^{(\hat B)}$. Suppose that $f\in\mathbb C(X^{[\mathbf{k}]})^{(\hat B)}$. Then the pull-back of $f$ is a $\hat B$-semiinvariant and $\mathbb T$-invariant function on $X\times(\mathbb C^{[\mathbf{k}]+\mathbf{1}})$. Denote by $(\varpi,\mu_1,...,\mu_r)$, where each $\mu_A=(\mu_{A,0},...,\mu_{A,{k_A}})\in\mathbb Z^{k_A+1}$, the $\hat B$-character of $f$. Then $\varpi\in\mathfrak M(G/H)$. Also, by $\mathbb T$-invariance we get
\begin{align}\label{chi-sum-mu}
\varpi_A+\chi_A=\mu_{A,0}+...+\mu_{A,k_A},~\forall1\leq A\leq r.
\end{align}
Conversely, for any given tuple $(\varpi,\mu_1,...,\mu_r)$ satisfying \eqref{chi-sum-mu},
$$\tilde f=\hat f_\varpi\prod_{A=1}^r\prod_{j_A=0}^{k_A}(z^{(A),j_A})^{\mu_{A,j_A}},$$
where $\hat f\in\mathbb C(X)^{(B)}_\varpi$, descends to a function in $\mathbb C(X^{[\mathbf{k}]})^{(\hat B)}_{(\varpi,\mu_1,...,\mu_r)}$. Hence the lattice of  $\mathbb C(X^{[\mathbf{k}]})^{(\hat B)}$ is
\begin{align}\label{M(X^[k])}
\mathfrak M(X^{[\mathbf{k}]})=\{(\varpi,\mu_1,...,\mu_r)\in\mathfrak M(G/H)\oplus\mathbb Z^{[\mathbf{k}]+\mathbf{1}}|~(\varpi,\mu_1,...,\mu_r)~\text{satisfies \eqref{chi-sum-mu}}\}.
\end{align}

Then we compute the $\hat B$-invariant divisors of $\mathfrak F(X^{[\mathbf{k}]})$ and there image in $\mathfrak N(X^{[\mathbf{k}]})$. Fix a set of basis $\{\lambda_{\alpha}\}_{\alpha=1}^{r_0}$ of $\mathfrak M(G/H)$. Then the vectors
\begin{align*}
\hat\lambda_\alpha&=(\lambda_\alpha;{\lambda_{\alpha}}_1,0,...,0,...;(\mu_{A,0}=){\lambda_{\alpha}}_A,0,...,(\mu_{A,k_A}=)0;...;{\lambda_{\alpha}}_r,0,...,0),~1\leq\alpha\leq r_0\\
\hat\mu_{A,i'_A}&=(O;0,...,0;(\mu_{A,0}=)-1,0,...,(\mu_{A,i'_{A}}=)1,0,...,0),~1\leq i'_A\leq k_A,~1\leq A\leq r
\end{align*}
form a basis of $\mathfrak M(X^{[\mathbf{k}]})$.

Recall the open covering $\{X^{[\mathbf{k}]}_{[i]}\}_{[i]}$. There are three types of $\hat B$-invariant divisors on $X^{[\mathbf{k}]}$:

\emph{Type-1.} For any $D\in\mathcal I_G(X)\cup\mathcal D(G/H)$, $D^{[\mathbf{k}]}:=D\times(\mathbb C^{[\mathbf{k}]+\mathbf{1}}\setminus{\{O\}})/\mathbb T$ is a $\hat B$-invariant divisor of $X^{[\mathbf{k}]}$. 
The image of $D^{[\mathbf{k}]}$ in $\mathfrak N(X^{[\mathbf{k}]})$ is characterized by
\begin{align*}
&D^{[\mathbf{k}]}(\hat\lambda_\alpha)=D(\lambda_\alpha),&&1\leq\alpha\leq r_0,\\
&D^{[\mathbf{k}]}(\hat\mu_{A,i'_A})=0,&&1\leq A\leq r,~1\leq i'_A\leq k_A.
\end{align*}

\emph{Type-2.} $D_{A,i'_A}:=\{z^{(A),i'_A}=0\}/\mathbb T,~1\leq A\leq r,~1\leq i'_A\leq k_A.$ Its image in $\mathfrak N(X^{[\mathbf{k}]})$ is characterized by
\begin{align*}
&D_{A,i'_A}(\hat\lambda_\alpha)=0,&&1\leq\alpha\leq r_0,\\
&D_{A,i'_A}(\hat\mu_{B,j'_B})=\delta_{AB}\delta_{i'_Aj'_B},&&1\leq B\leq r,~1\leq j'_B\leq k_B.
\end{align*}

\emph{Type-3.} $D_{A,0}:=\{z^{(A),0}=0\}/\mathbb T,~1\leq A\leq r.$ Its image in $\mathfrak N(X^{[\mathbf{k}]})$ is characterized by
\begin{align*}
&D_{A,0}(\hat\lambda_\alpha)={\lambda_\alpha}_A,&&1\leq\alpha\leq r_0,\\
&D_{A,0}(\hat\mu_{B,j'_B})=-\delta_{AB},&&1\leq B\leq r,~1\leq j'_B\leq k_B.
\end{align*}


Finally we determine
all $\hat B$-semiinvariant sections of $(X^{[\mathbf{k}]},(L^{[\mathbf{k}]})^p)$ for any $p\in\mathbb N$. Suppose that $s\in{\rm H}^0(X^{[\mathbf{k}]},(L^{[\mathbf{k}]})^p)^{(\hat B)}$. Then the pull-back $\bar s$ of $s$ on $X\times((\mathbb C^{[\mathbf {k}]+\mathbf{1}})\setminus\{O\})$ is a $\hat B$-semiinvariant and $\mathbb T$-invariant section.

Suppose that the $\hat B$-character associated to $s$ is $(\varpi,\mu_1,...,\mu_r)$. Then so is $\bar s$. Let $\vartheta=(e^{a_1},...,e^{a_r})\in\mathbb T\cong(\mathbb C^*)^r$ with each $a_A\in\mathbb C$. Then for any $(x;z^{(A)})\in X\times(\mathbb C^{[\mathbf{k}]+\mathbf{1}}\setminus\{O\})$, by $\mathbb T$-invariance,
\begin{align*}
(\vartheta\cdot \bar s)(x,z^{(A)})=&\vartheta\cdot(\bar s(\iota(\vartheta)x;e^{-a_A}z^{(A)}))\\
=&\vartheta\cdot(\vartheta^{-1}\cdot(\varpi|_\mathfrak t+\chi)(\iota(\vartheta))\prod_{A=1}^r\mu_A(e^{a_A})\bar s)\\
=&(\varpi|_\mathfrak t+\chi)(\iota(\vartheta))\prod_{A=1}^r\mu_A(e^{a_A})\cdot\bar s(x;z^{(A)}).
\end{align*}
Hence the tuple $(\varpi,\mu_1,...,\mu_r)$ satisfies \eqref{chi-sum-mu}.

By restricting $s$ on each $X^{[\mathbf{k}]}_{[i]}$, it holds
\begin{align*}
s(x;z^{(A)})=\prod_{A=1}^r\prod_{j_A=0,j_A\not=i_A,z^{(A),j_A}\not=0}^{k_A}\mu_{A,j_A}(z^{(A),j_A})s(x;{\rm sgn}(|z^{(A),j_A}|)).
\end{align*}
Since $s$ is holomorphic,
\begin{align}
D^{[\mathbf{k}]}(s)&\geq0,~D\in\mathcal I_G(X)\cup\mathcal D(G/H),\label{val-B-bd-divisor}\\
D_{A,i_A}(s)&\geq0,~1\leq A\leq r,~0\leq i_A\leq k_A.\label{val-T-bd-divisor}
\end{align}
The relation \eqref{val-B-bd-divisor} implies that $\varpi\in\Delta_{L,p}$. The relation \eqref{val-T-bd-divisor} implies $\mu_{A,i_A}\geq0$ for all $1\leq A\leq r,$ and $0\leq i_A\leq k_A$. Note that by our assumption on the lifting of $T$-action on $L$, each $\lambda_A+\chi_A\geq0$ and \eqref{Xk-Lk-polytope} is a convex polytope. Combining with \eqref{chi-sum-mu} we get \eqref{Xk-Lk-polytope}.

\end{proof}

\subsection{$G$-equivariant normal test configurations}
Let $(X,L)$ be a polarized $G$-spherical variety. The $G$-equivariant (ample) normal test configurations of $(X,L)$ have been classified by \cite{Del-2020-09,LL-202206}.
\begin{prop}\label{ZTC-classification}
Let $(X,L)$ be a polarized $G$-spherical variety with $\Delta_+(L)$ its moment polytope. Then any $G$-equivariant normal test configuration $(\mathcal X,\mathcal L)$ of $(X,L)$ is a polarized $G\times\mathbb C^*$-spherical variety. Moreover, there is a rational, concave, piecewise linear function $f:\Delta_+(L)\to\mathbb R_+$ with $\nabla f\in\mathcal V(G/H)$ so that the moment polytope $\Delta_+(\mathcal L)$ of $(\mathcal X,\mathcal L)$ is given by
\begin{align}\label{Delta+mathcal-L}
\Delta_+(\mathcal L)=\{(y,t)\in\Delta_+(L)\times\mathbb R|~0\leq t\leq f(y)\}.
\end{align}
\end{prop}
Proposition \ref{ZTC-classification} was first proved by \cite[Theorem 4.1]{Del-2020-09} based on \cite{AB2,AK}. It can also be derived from a general classification result \cite[Theorem 3.4]{LL-202206} of $\mathbb R$-equivariant test configurations (cf. \cite[Remark 3.6]{LL-202206}).

Suppose that the function $f$ corresponding to $(\mathcal X,\mathcal L)$ is given by
\begin{align}\label{f-tc-func}
f(\lambda)=\min_{a=1,...,N_f}\{l_a(\lambda):=C_a+\Lambda_a(\lambda)\}:\Delta_+(L)\to\mathbb R_+,
\end{align}
where each $C_a\in\mathbb Q$, $\Lambda_a\in\mathfrak N_\mathbb Q(G/H)$. Here we assume that the set $\{l_a(y)\}_{a=1}^{N_f}$ is minimal so that deleting any $l_a$ will change $f$. Also we associate to each $\Lambda_a$ a number $m_a\in\mathbb N_+$ which is the minimal positive integer so that $m_a\Lambda_a\in\mathfrak N(G/H)$. Then $(\mathcal X,\mathcal L)$ is a $G\times\mathbb C^*$-spherical variety. The invariant $G\times\mathbb C^*$-divisors of $\mathcal X$ are
\begin{itemize}
\item $\hat D=\overline{D\times\mathbb C^*}$, $D\in\mathcal I_G(X)$;
\item $\mathcal X_{0,a}$, the primitive $G\times \mathbb C^*$-invariant divisor corresponding to the $a$-th piece of $f$;
\item $\mathcal X_\infty\cong X$, the divisor corresponding to the $\Delta_+(L)\times\{0\}$, which is the fibre of $\mathcal X$ at $\infty\in\mathbb P^1$.
\end{itemize}
The colours are
\begin{itemize}
\item $\hat D=\overline{D\times\mathbb C^*}$, $D\in\mathcal D(G/H)$.
\end{itemize}
Also, the central fibre
\begin{align}\label{X0}
\mathcal X_0=\sum_{a=1}^{N_f}m_a\mathcal X_{0,a}.
\end{align}
We directly conclude that $\mathcal X_0$ is reduced if and only if all $m_a=1$, or equivalently, each $\Lambda_a$ in \eqref{f-tc-func} is integral; it has one irreducible component if and only if $f$ is affine. Moreover, $\mathcal X_0$ is normal if and only if $f$ is affine and has integral coefficients (cf. \cite[Corollary 3.9]{LL-202206}).

Assume that the exponent of $(\mathcal X,\mathcal L)$ is $m_0$. Choose a $B$-semiinvariant section $s$ of $L^{m_0}$ with $B$-character $\varpi_s$ so that its divisor
\begin{align}\label{divisor-L}
\mathfrak d_s=m_0(\sum_{D\in\mathcal I_G(X)\cup\mathcal D(G/H)}C_DD),
\end{align}
then it is direct to derive the following Lemma from \cite[Section 17.4]{Timashev-book},
\begin{lem}\label{mathcal-L} There is a $B\times \mathbb C^*$-semiinvariant section $\hat s$ of $\mathcal L$ with $B\times \mathbb C^*$-character $(\varpi_s,0)$ whose divisor
\begin{align*}
\hat{\mathfrak d}_s=m_0(\sum_{D\in\mathcal I_G(X)\cup\mathcal D(G/H)}C_D\hat D+\sum_{a=1}^{N_f}m_a(C_a+\Lambda_a(\varpi_s))\mathcal X_{0,a}).
\end{align*}
\end{lem}
When $X$ is $\mathbb Q$-Fano, take $L=-K_X$ and consider the $G$-equivariant normal test configuration $(\mathcal X,\mathcal L)$ of $(X,L)$. There is a $B$-stable Weil divisor
$$-K_X=\sum_{D\in\mathcal I_G(X)} D+\sum_{D\in\mathcal D(G/H)}n_D D,$$
and $-k_0K_X$ is Cartier for sufficiently divisible $k_0\in\mathbb N_+$. Also there is a canonical $B$-semiinvariant section $s_0$ of $-k_0K_X$ with weight $k_0\kappa_P$. It follows that the $B\times \mathbb C^*$-character of the $\mathbb C^*$-invariant rational section $\bar s_0$ of $k_0\mathcal L$ induced by $s_0$ is $(k_0\kappa_P,0)$. Also,
\begin{align*}
-K_\mathcal X=\sum_{D\in\mathcal I_G(X)}\hat D+\sum_{D\in\mathcal D(G/H)}n_D\hat D+\sum_{a=1}^{N_f}\mathcal X_{0,a}+\mathcal X_\infty,
\end{align*}
and consequently, by Lemma \ref{mathcal-L},
\begin{align}\label{-K-log}
-K^{\log}_{\mathcal X/\mathbb P^1}=&\sum_{D\in\mathcal I_G(X)}\hat D+\sum_{D\in\mathcal D(G/H)}n_D\hat D\notag\\
=&\mathcal L-\sum_{a=1}^{N_f}m_a(C_a+\Lambda_a(\kappa_P))\mathcal X_{0,a},
\end{align}
whose $B\times \mathbb C^*$-character is also $(\kappa_P,0)$.

\section{The $g$-weighted non-Archimedean functionals of $G$-equivariant normal test configurations}
Let $X$ be a $\mathbb Q$-Fano $G$-spherical variety, which is a spherical embedding of some $G/H$. Let $(\mathcal X,\mathcal L)$ be any $G$-equivariant normal test configuration of $(X,-K_X)$. Then it is a spherical embedding of $G\times\mathbb C^*/H\times\{e\}$. From Section 3.1.2 we know that ${\rm Aut}_{G\times \mathbb C^*}(\mathcal X)=N_G(H)/H\times\mathbb C^*$ and $(\mathcal X,\mathcal L)$ is automatically $G\times {\rm Aut}_G^0(X)$-equivariant.

Denote by $\Delta_+$ the moment polytope $\Delta_+(-K_X)$ with respect to the canonical lifting of the $G$-action for short. In the remaining, we will compute the $g$-weighted non-Archimedean functionals. Suppose that we have a lifting of $T$-action on $L$ with respect to the character $\chi$. Then $\Delta=\mathfrak r_\chi(\Delta_+)$ with $\mathfrak r_\chi(\cdot)=\mathfrak r(\cdot)+\chi$, and we may identify $g$ with its pull-back through \begin{align}\label{g-pullback}
(\mathfrak r_\chi)^*g:\Delta_+(L)\to\mathbb R
\end{align}
so that $g$ can be identified with a function on $\Delta_+(L)$.

Denote by $\{\xi_A\}_{A=1}^r$ a basis of $\mathfrak N(T)$-the lattice of one-parameter subgroups of $T\cong(\mathbb C^*)^r$. From the above discussion we may choose suitable coordinates for $\lambda=(\lambda_1,...,\lambda_{r_G})\in\mathfrak X_\mathbb R(B)$, where $r_G$ is the rank of $G$, so that the first $r$ coordinates $\lambda_A=\xi_A^*,~A=1,..,r$. Then $g(\mathfrak r_\chi(\lambda))=\hat g(\lambda_1,...,\lambda_{r})$ for some function $\hat g$ on $\Delta_+$ which depends only on the first $r$-arguments of $\lambda$.
We have
\begin{lem}\label{Vg-lem}
Let $X$ be a $\mathbb Q$-Fano $G$-spherical variety with moment polytope $\Delta_+$. For a positive continuous weight $g$ on $\Delta$, it holds
\begin{align}\label{Vg-monomial}
V_g=n!\int_{\Delta_+}g(\mathfrak r_\chi(\lambda))\pi(\lambda)d\lambda.
\end{align}
\end{lem}
\begin{proof}
Denote by $\theta$ the points in $\Delta$. As in \cite[Section 2]{Han-Li-KRS},
\begin{align*}
V_g=\int_\Delta g(\theta)\left({{\mathbf m}_{\omega_\phi}}_*\frac{\omega_\phi^n}{n!}\right)(\theta)=&\int_{\mathfrak r_\chi(\Delta_+)}g(\mathfrak r_\chi(\lambda))(\mathfrak r_\chi\circ\mu_{\omega_\phi})_*\frac{\omega_\phi^n}{n!}\\
=&n!\int_{\Delta_+}g(\mathfrak r_\chi (\lambda))\pi(\lambda)d\lambda,
\end{align*}
which gives \eqref{Vg-monomial}.
\end{proof}
In the following we compute the $g$-weighted non-Archimedean functionals of $G$-equivariant normal test configurations of $(X,-K_X)$. Fix the canonical lifting $\sigma_0$ of the $T$-acgion on $L$ with $\chi_0$ the associated $T$-character, and $\Delta_0$ the corresponding moment polytope.
Denote by $F_D$ the facet of $\Delta_+$ that corresponds to the $B$-invariant divisor $D$. Set
\begin{align}\label{coef-bdry-Delta}
\tau_D=\left\{\begin{aligned}\frac{1-\kappa_P(u_D)}{|u_D|},~&\text{if}~D\in\mathcal I_G(X),\\
\frac{n_D-\kappa_P(\varrho_D)}{|\varrho_D|},~&\text{if}~ D\in\mathcal D(G/H).   \end{aligned}\right.
\end{align}
Denote by $d\sigma_0$ the Lebesgue measure of $\partial\Delta_+$. Define a measure $d\sigma$ of $\partial\Delta_+$ so that $d\sigma|_{F_D}=\tau_Dd\sigma_0|_{F_D}$ on each facet $F_D$.

We have:
\begin{prop}\label{NA-functional}
Let $X$ be a $\mathbb Q$-Fano $G$-spherical variety with moment polytope $\Delta_+$. Let $\Delta_0$ be the moment polytope of the canonical lifting $\sigma_0$ of the $T$-action on $L$, and the weight $g\in C^0(\Delta_0)$. Then for the $G$-equivariant normal test configuration $(\mathcal X,\mathcal L)$ of $(X,-K_X)$ corresponding to the function $f$, it holds
\begin{align}
{\rm E}_g^{\rm NA}(\mathcal X,\mathcal L)=&\frac1{V_g}\int_{\Delta_+}fg(\mathfrak r_{\chi_0} (\lambda))\pi(\lambda) d\lambda,\label{ENA-monomial}\\
{\rm J}_g^{\rm NA}(\mathcal X,\mathcal L)=&\frac1{V_g}\int_{\Delta_+}(\max f-f)g(\mathfrak r_{\chi_0} (\lambda))\pi(\lambda) d\lambda,\label{JNA-monomial}\\
{\rm D}_g^{\rm NA}(\mathcal X,\mathcal L)=&f(\kappa_P)-\frac1{V_g}\int_{\Delta_+}fg(\mathfrak r_{\chi_0} (\lambda))\pi(\lambda) d\lambda \label{DNA-Fano-monomial}\\
{\rm M}_g^{\rm NA}(\mathcal X,\mathcal L)=&-\frac1{V_g}\int_{\Delta_+}\langle\nabla f,\lambda -\kappa_P\rangle g(\mathfrak r_{\chi_0} (\lambda))\pi(\lambda) d\lambda.\label{MNA-Fano-monomial}
\end{align}

\end{prop}

\begin{proof}

\emph{Step-1. $g$ is a monomial \eqref{monom-g} of $\theta=\mathfrak r_{\chi_0}(\lambda)\in\Delta_0$}. We will mainly use \cite[Theorem 18.8]{Timashev-book} to compute the intersection numbers. 
Denote by $$\Sigma_m(c)=\{\mu_A\in\mathbb R_{\geq0}^{m+1}|c=\mu_{0}+...+\mu_{m}\}$$ the $c$-dilation of the standard $m$-dimensional simplex. Then its normalized volume is $\frac{c^{m}}{m!}$.

The line bundle $\mathcal L^{\sigma_0[\mathbf k]}$ may not be ample in general. In order to use the intersection formula \cite[Theorem 18.8]{Timashev-book}, we consider another lifting $\sigma$ of the $T$-action on $L$ with a suitable character $\chi_0+\chi$ so that $\mathcal L^{\sigma[\mathbf j]}$ for each $\mathbf0\leq\mathbf j\leq\mathbf k$ is ample. Take $\sigma_1=\sigma_0$ and $\sigma_2=\sigma$ in \eqref{change-of-L}, we can always choose such a $\chi$.

To prove \eqref{ENA-monomial}, by Proposition \eqref{change-of-NA-func},
\begin{align}\label{change-of-ENA}
{\rm E}_g^{\rm NA}(\mathcal X,\mathcal L)&={\rm E}_{g(\cdot-\chi)}^{\rm NA}(\mathcal X,\mathcal L^\sigma)\notag\\
=&\sum_{\mathbf0\leq\mathbf i\leq\mathbf k}\frac{(\mathbf k-\mathbf i)!C_{\mathbf k}^{\mathbf k-\mathbf i}}{(n+|\mathbf k-\mathbf i|+1)!V_g}(-\chi)^{\mathbf i}(\mathcal L^{\sigma[\mathbf k-\mathbf{i}]})^{n+|\mathbf k-\mathbf i|+1}.
\end{align}
Recall that $(\mathcal X,\mathcal L^\sigma)$ has moment polytope \eqref{Delta+mathcal-L}. Note that the moment polytope $\Delta_+(\mathcal L)$ is determined by the lifting of the $G$-action on $L$ rather that the $T$-action, it leaves unchanged when replacing $\sigma_0$ by $\sigma$. But the map
$$\mathfrak r_{\chi_0+\chi}(\cdot)=\mathfrak r_{\chi_0}(\cdot)+\chi:\Delta_+(\mathcal L)\to\Delta_{\sigma},$$
translates the $T$-moment polytope of $\mathcal L^{\sigma_0}$ by $\chi$. Applying Proposition \ref{Xk-Lk-structure} and \cite[Theorem 18.8]{Timashev-book} to $\mathcal L^\sigma$, for each $\mathbf j:=\mathbf k-\mathbf i$, we have
\begin{align}\label{L-j-term}
\frac{\mathbf j!}{(n+|\mathbf j|+1)!V_g}(\mathcal L^{\sigma[\mathbf j]})^{n+|\mathbf j|+1}&=\frac{\mathbf j!}{V_g}\int_{\Delta_+(\mathcal L^{[\mathbf{j}]})}\pi(\lambda)d\lambda\wedge dt\wedge d\mu\notag\\
&=\frac{\mathbf j!}{V_g}\int_{\Delta_+(\mathcal L)}g(\mathfrak r_{\chi_0+\chi}(\lambda))\pi(\lambda)d\lambda\wedge dt\cdot\prod_{A=1}^r\frac1{j_A!}\notag\\
&=\frac1{V_g}\int_{\Delta_+}fg(\mathfrak r_{\chi_0}(\lambda)+\chi)\pi(\lambda)d\lambda.
\end{align}
Here the factor $\frac1{j_A!}$ in the second line is the normalized volume of $\Sigma_{j_A}(1)$, and in the last line we used \eqref{Delta+mathcal-L}. Plugging the above relation into \eqref{change-of-ENA} we get \eqref{ENA-monomial}. The relation \eqref{DNA-Fano-monomial} follows from \eqref{ENA-monomial},
$${\rm D}_g^{\rm NA}(\mathcal X,\mathcal L)={\rm L^{NA}}(\mathcal X,\mathcal L)-{\rm E}_g^{\rm NA}(\mathcal X,\mathcal L)$$
 and (cf. \cite[Section 5]{LL-202206}),
$${\rm L^{NA}}(\mathcal X,\mathcal L)=f(\kappa_P).$$

Now we turn to \eqref{JNA-monomial}. The $g$-weighted non-Archimedean J-functional
\begin{align*}
{\rm J}_{g}^{\rm NA}(\mathcal F):=\frac 1{(-K_X)^{n}}\mathcal L\cdot (-K_X)_{\mathbb{P}^1}^{ n}-{\rm E}_{g}^{\rm NA}(\mathcal F).
\end{align*}
By \eqref{ENA-monomial} it suffices to compute the first term on the right-hand-side. To compute the intersection number, we use the method of \cite[Section 18]{Timashev-book}. Note that for any $\epsilon>0$, the Newton polytope of the ample line bundle $\mathcal L'_\epsilon:=\epsilon \mathcal L+ (-K_X)_{\mathbb{P}^1}$ is $\mathcal P_\epsilon:=\epsilon \mathcal P_\mathcal L+(\Delta_+\times\{0\})$. By \cite[Corollary 18.28]{Timashev-book},
$$\frac1{(n+1)!}{\mathcal L'}_\epsilon^{n+1}=\int_{\mathcal P_\epsilon}\pi d\lambda\wedge dt=\epsilon\max f\cdot\int_{P_+}\pi dy+O(\epsilon^2),~\epsilon\to0^+.$$
Hence
$$\mathcal L\cdot L_{\mathbb{P}^1}^{ n}=n!\left.\frac{d}{d\epsilon}\right|_{\epsilon_0}{\mathcal L'}_\epsilon^{n+1}=n!\max f\cdot\int_{P_+}\pi dy=L^{ n}\cdot\max f.$$
Combining with \eqref{ENA-monomial} we get \eqref{JNA-monomial}.

Finally we prove \eqref{MNA-Fano-monomial}. In view of \eqref{ENA-monomial} and
\begin{align}\label{MNA-def-monomial}
{\rm M}_g^{\rm NA}(\mathcal X,\mathcal L)=\frac{\mathbf k!}{V_g(n+|\mathbf k|)!}(K^{\log}_{\mathcal X/\mathbb P^1})^{[\mathbf{k}]}(\mathcal L^{\sigma_0[\mathbf{k}]})^{n+|\mathbf k|}+(n+|\mathbf{k}|){\rm E}_g^{\rm NA}(\mathcal X,\mathcal L),
\end{align}
it suffices to compute the term $(K^{\log}_{\mathcal X/\mathbb P^1})^{[\mathbf{k}]}(\mathcal L^{\sigma_0[\mathbf{k}]})^{n+|\mathbf k|}$.

By \eqref{K_X/P-log-k-Weil-div} and \eqref{-K-log},
\begin{align}\label{K^log-L-term -j}
\frac{\mathbf k!(K^{\log}_{\mathcal X/\mathbb P^1})^{[\mathbf{k}]}(\mathcal L^{\sigma_0[\mathbf{k}]})^{n+|\mathbf k|}}{(n+|\mathbf k|)!V_g}
=&-\frac{\mathbf k!}{(n+|\mathbf k|)!V_g}(\mathcal L^{\sigma_0[\mathbf{k}]})^{n+|\mathbf k|+1}\notag\\
 &+\frac{\mathbf k!}{(n+|\mathbf k|)!V_g}\sum_{a=1}^{N_f}m_a(C_a+\Lambda_a(\kappa_P))\mathcal X_{0,a}^{[\mathbf k]}(\mathcal L^{\sigma_0[\mathbf{k}]})^{n+|\mathbf k|}
\end{align}
It remains to deal with the second term. Choose the lifting $\sigma=\sigma_0+\chi$ of the $T$-action as before, we have
\begin{align}\label{change-of-K-log-term}
&\frac{\mathbf{k}!}{(n+|\mathbf k|)!}\mathcal X_{0,a}^{[\mathbf k]}(\mathcal L^{\sigma_0[\mathbf{k}]})^{n+|\mathbf k|}\notag\\
=&\sum_{\mathbf0\leq\mathbf{(\mathbf{k}-\mathbf{i})!}\leq\mathbf k}\frac{1}{(n+|\mathbf{k-i}|)!}C_{\mathbf k!}^{\mathbf i!}(-\chi)^{\mathbf i}(\mathcal X_{0,a})^{[\mathbf{k}-\mathbf{i}]}(\mathcal L^{\sigma[\mathbf{k}-\mathbf{i}]})^{\cdot(n+|\mathbf{k}-\mathbf{i}|)}.
\end{align}
For each $\mathbf 0\leq\mathbf j:=\mathbf k-\mathbf i\leq\mathbf k$,
\begin{align}\label{K^log-L-term -j2}
\mathcal X_{0,a}^{[\mathbf{j}]}(\mathcal L^{\sigma[\mathbf{j}]})^{n+|\mathbf j|}=(\mathcal L^{\sigma[\mathbf{j}]}|_{\mathcal X_{0,a}^{[\mathbf{j}]}})^{n+|\mathbf j|}=((\mathcal L|_{\mathcal X_{0,a}})^{\sigma[\mathbf{j}]})^{n+|\mathbf j|}.
\end{align}
Note that each $\mathcal X_{0,a}$ is a $G\times \mathbb C^*$-spherical subvariety of $\mathcal X$ that corresponds to the coloured cone $(\mathbb Q_{\geq0}(m_a\Lambda_a,-m_a),\emptyset)$ in $\mathfrak N_\mathbb Q(G/H)\oplus\mathbb Z$. Also, take any integral point $\tilde\lambda$ of $\Delta_+(\mathcal L)$ that lies in the $a$-th piece of the graph of $f$, it corresponds to a section $\tilde s\in{\rm H}^0(\mathcal X,\mathcal L)_{\tilde\lambda}^{(B\times\mathbb C^*)}$ that does not vanish on $\mathcal X_{0,a}$. Hence $\tilde s|_{\mathcal X_{0,a}}$ gives a section of the ample line bundle $\mathcal L|_{\mathcal X_{0,a}}$. On the other hand, by \cite[Theorem 15.14]{Timashev-book}, the lattice of $B\times\mathbb C^*$-semiinvariant rational functions
$$\mathfrak M(\mathcal X_{0,a})=\mathfrak M(G/H\times\mathbb C^*)\cap(m_a\Lambda_a,-m_a)^\perp,$$
and each $\tilde f_{\tilde \mu}\in\mathbb C(\mathcal X_{0,a})^{(B\times\mathbb C^*)}_{\tilde\mu}$ with $\tilde\mu\in\mathfrak M(\mathcal X_{0,a})$ is the restriction $f_{\tilde \mu}|_{\mathcal X_{0,a}}$ of some $f_{\tilde \mu}\in\mathbb C(G/H\times\mathbb C^*)^{(B\times\mathbb C^*)}_{\tilde\mu}$. Thus any $B\times\mathbb C^*$-semiinvariant rational section of $\mathcal L|_{\mathcal X_{0,a}}$ is the restriction of $f_{\tilde \mu}\tilde s$ for some $f_{\tilde \mu}\in\mathbb C(G/H\times\mathbb C^*)^{(B\times\mathbb C^*)}_{\tilde\mu}$ with $\tilde\mu\in\mathfrak M(G/H\times\mathbb C^*)$ perpendicular to $(m_a\Lambda_a,-m_a)$. Using \cite[Theorem 1.2]{Ga-Ho-datum} (or essentially, \cite[Theorem 2.8 (d)]{Ga-Ho-datum}), we see that $(f_{\tilde \mu}\tilde s)|_{\mathcal X_{0,a}}$ is holomorphic on $\mathcal X_{0,a}$ if and only if it is holomorphic on $\mathcal X$. Thus, denote by $\Omega_a$ the domain in $\Delta_+$ where $f=l_a$ and $F_a$ the graph of $l_a$ over $\Omega_a$, we conclude that
$${\rm H^0}(\mathcal X_{0,a},\mathcal L^k_{\mathcal X_{0,a}})\cong\oplus_{\lambda\in\overline{k\Omega_a}\cap\mathfrak M(G/H),kC_a+\Lambda_a(\lambda)\in\mathbb Z}V_{(\lambda,kC_a+\Lambda_a(\lambda))},$$
where each $V_{\lambda,m}$ is considered as an irreducible $G\times \mathbb C^*$-representation with highest weight $(\lambda,m)$. Clearly,
$$\dim V_{\lambda,m}=V_\lambda,~\forall\lambda\in\mathfrak X_+(G)~\text{and}~m\in\mathbb Z,$$
and by \cite[Theorem 18.8]{Timashev-book},
\begin{align*}
((\mathcal L|_{\mathcal X_{0,a}})^{\sigma[\mathbf{j}]})^{n+|\mathbf j|}
=&\frac{(n+|\mathbf j|)!}{\mathbf j!}\int_{F_a}\frac1{m_a\sqrt{1+|\Lambda_a|^2}}g(\mathfrak r_{\chi_0}(\lambda)+\chi)\pi(\lambda)d\sigma_0\\
=&\frac{(n+|\mathbf j|)!}{m_a\mathbf j!}\int_{\Omega_a}g(\mathfrak r_{\chi_0}(\lambda)+\chi)\pi(\lambda)d\lambda,
\end{align*}
where $d\sigma_0$ denotes the standard Lebesgue measure on $F_a$. Plugging the above equality into \eqref{K^log-L-term -j2}, also note that
$$C_a+\Lambda_a(\kappa_P)=f+\langle\kappa_P-\lambda,\nabla f\rangle,~\forall\lambda\in\Omega_a,$$
we get
\begin{align*}
\frac{\mathbf j!\mathcal X_{0,a}^{[\mathbf{j}]}(\mathcal L^{\sigma[\mathbf{j}]})^{n+|\mathbf j|}}{(n+|\mathbf j|)!V_g}
=&\frac1{m_aV_g}\int_{\Delta_+}\langle\kappa_P-\lambda,\nabla f\rangle g(\mathfrak r_{\chi_0}(\lambda)+\chi)\pi(\lambda)d\lambda\\
&+\frac{\mathbf j!}{m_a(n+|\mathbf j|+1)!V_g}(\mathcal L^{\sigma[\mathbf j]})^{n+|\mathbf j|+1},
\end{align*}
where in the last line we used \eqref{L-j-term}. Plugging this relation into \eqref{change-of-K-log-term} and combining with \eqref{K^log-L-term -j}, we have
\begin{align*}
\frac{\mathbf k!(K^{\log}_{\mathcal X/\mathbb P^1})^{[\mathbf{k}]}(\mathcal L^{\sigma_0[\mathbf{k}]})^{n+|\mathbf k|}}{V_g(n+|\mathbf k|)!}
=&-\frac{\mathbf k!(n+|\mathbf k|)}{(n+|\mathbf k|+1)!V_g}(\mathcal L^{\sigma_0[\mathbf{k}]})^{n+|\mathbf k|+1}\notag\\
 &+\frac1{V_g}\int_{\Delta_+}\langle\kappa_P-\lambda,\nabla f\rangle g(\mathfrak r_{\chi_0}(\lambda))\pi(\lambda)d\lambda,
\end{align*}
and we get \eqref{MNA-Fano-monomial} by using \eqref{MNA-def-monomial} and \eqref{ENA-monomial}.


\emph{Step-2.The case of a general $C^0$-weight $g$.} Now we turn to the case when $g$ is a general $C^0$-function on $\Delta_+(L)$. 
Suppose that $g$ is given by \eqref{g-polynomial}.
Recall the construction in Section 2.1, \emph{Step-2}. Since the functionals $V_g\cdot{\rm N}_g^{\rm NA}(\mathcal X,\mathcal L), {\rm N}\in\{\rm E,J,D,M\}$ are all linear in $g$, clearly Proposition \ref{NA-functional} holds for polynomial $g$. For general $C^0$-function $g$, we can approximate it by polynomials in the $C^0$-topology and then take limit. We conclude the Proposition.

\end{proof}

\begin{rem}
When $X$ is $\mathbb Q$-Fano and $L=-K_X$, it is proved in \cite{Ga-Ho} that $\Delta_+$ is a $\mathbb Q$-reflexive polytope. Denote by $\nu$ the unit outer normal vector of $\partial \Delta_+$. Then by\eqref{coef-bdry-Delta} we have $$\pi d\sigma=\langle \lambda,\nu\rangle\pi d\sigma_0.$$

Taking integration by parts and using homogeneity of $\pi(\lambda)$,
\begin{align*}
\int_{\Delta_+}fg\pi d\sigma=&\int_{\Delta_+}\langle\nabla(fg\pi),\lambda\rangle d\lambda+r_0\int_{\Delta_+}fg\pi d\lambda\\
=&\int_{\Delta_+}\langle\nabla f,\lambda\rangle g\pi d\lambda+\int_{\Delta_+}f\langle\nabla g,\lambda\rangle\pi d\lambda\\&+(n-r_0)\int_{\Delta_+}fg\pi d\lambda+r_0\int_{\Delta_+}fg\pi d\lambda,
\end{align*}
and we get another expression
\begin{align}\label{MNA-monomial}
{\rm M}^{\rm NA}_g(\mathcal X,\mathcal L):=&-\frac1{V_g}\left(\int_{\partial\Delta_+}fg(\mathfrak r_{\chi_0}(\lambda))\pi d\sigma-\int_{\Delta_+}\kappa_P(\nabla f)g(\mathfrak r_{\chi_0}(\lambda))\pi d\lambda\right.\notag\\
&\left.-n\int_{\Delta_+}fg(\mathfrak r_{\chi_0}(\lambda))\pi d\lambda-\int_{\Delta_+}f\langle\lambda,\nabla g(\mathfrak r_{\chi_0}(\lambda))\rangle\pi d\lambda\right).
\end{align}
\end{rem}




We have the following inequality:
\begin{prop}\label{Ding-Mabuchi}
Let $X$ be a $\mathbb Q$-Fano $G$-spherical variety and $(\mathcal X,\mathcal L)$ a $G$-equivariant normal test configuration of $(X,-K_X)$. Let $g>0$ be a $C^1$-function. Then
\begin{align}\label{MgeqD}
{\rm M}_g^{\rm NA}(\mathcal X,\mathcal L)\geq{\rm D}_g^{\rm NA}(\mathcal X,\mathcal L),
\end{align}
and the equality holds if and only if the central fibre $\mathcal X_0$ of $(\mathcal X,\mathcal L)$ has only one irreducible component.
\end{prop}

\begin{proof}
Let $f$ be the piecewise linear concave function associated to $(\mathcal X,\mathcal L)$. By concavity,
$$-f(\lambda)+f(\kappa_P)\leq\langle\nabla(-f),\lambda-\kappa_P\rangle,$$
with the equality holds if and only if $f$ is affine. Plugging this into \eqref{MNA-Fano-monomial} (in the sense of Proposition \ref{NA-functional}),
\begin{align*}
{\rm M}_g^{\rm NA}(\mathcal X,\mathcal L)\geq\frac1{V_g}\int_{\Delta_+}(f(\kappa_P)-f(\lambda))g\pi d\lambda={\rm D}_g^{\rm NA}(\mathcal X,\mathcal L),
\end{align*}
with the equality holds if and only if $f$ is affine. The Proposition then follows from \eqref{DNA-Fano-monomial} and Proposition \ref{NA-functional}.
\end{proof}

\begin{rem}
By \cite[Section 5.1]{LL-202206}, one concludes that \eqref{MgeqD} holds if and only if $(\mathcal X,\mathcal L)$ is special after a possible base change.
\end{rem}

\section{The $g$-modified Futaki invariant of equivariant test configurations}

In this section we first compute the $g$-modified Futaki invariant of an equivariant test configuration when $g$ is smooth in a neighbourhood of $\Delta$. By Theorem \ref{Fut-g-smooth-thm}, the $g$-modified Futaki invariant can be computed using a similar argument as in Section 4. However, as we have assumed that $g\in C^\infty_0(\mathbb R^r)$, we can apply the Euler-Maclaurin formula  \cite[Theorem 4.1]{Guillemin-Sternberg-2007} to give a direct computation of the $g$-modified Futaki invariant according to Definition \ref{g-mod-Fu-def}. The advantage of the following argument is that we need not to do the computation step-by-step from the case of monomial $g$ to smooth $g$.

\begin{prop}\label{mod-futaki}
Let $(X,L)$ be a $\mathbb Q$-Fano $G$-spherical variety and $T\subset{\rm Aut}_G(X)$ be an $r$-dimensional torus. Suppose that $(\mathcal X,\mathcal L)$ the $G\times T$-equivariant test configuration of $(X,L)$ that is associated to the concave function $f$. Let $\Delta\subset\mathbb R^r$ be the moment polytope of the $T$-action with respect to the canonical lifting and $g\in C^\infty_0(\mathbb R^r)$. Then the $g$-modified Futaki invariant of $(\mathcal X,\mathcal L)$ is well-defined and satisfies
\begin{align}
2\frac{\int_{\Delta_+}\pi d\lambda}{\int_{\Delta_+}g\pi d\lambda}{\rm Fut}_g(\mathcal X,\mathcal L)=&-\int_{\partial\Delta_+}fg\pi d\sigma+n\int_{\Delta_+}fg\pi d\lambda\notag\\
&+\int_{\Delta_+}\kappa_P(\nabla f)g\pi d\lambda+\int_{\Delta_+}f\langle\lambda,\nabla g\rangle\pi d\lambda\notag\\
&+\sum_{a=1}^{N_f}(1-\frac1{m_a})\int_{\Omega_a}g\pi(\lambda)d\lambda,\label{mod-futaki-eq}
\end{align}
where the measure $d\sigma$ is $d\sigma|_{F_D}=\tau_Dd\sigma_0|_{F_D}$ given by \eqref{coef-bdry-Delta}.
\end{prop}

\begin{proof}
Let $(\mathcal X,\mathcal L)$ be the normal test configuration associated to $f$ defined in Proposition \ref{ZTC-classification}. Then up to a uniform translation, the eigenvalue of the $\exp(t\xi)$- and $\exp(t\Lambda)$-actions on the isotypic factor $V_\lambda\subset {\rm H}^0(\mathcal X_0,-k\mathcal L_0)$ are $e^{\xi(\lambda)}$ and $e^{[kf(\lambda/k)]}$, respectively (cf. \cite[Section 3]{LL-202206}). On the other hand, each isotypic factor $V_\lambda$ corresponds to a unique isotypic factor of ${\rm H}^0(X_,-kK_X)$ of the same $\lambda$ (cf. \cite{Popov-1986,Del3}). Thus, we have
\begin{align*}
S^{(g)}_{1|k}(\mathcal X,\mathcal L)=&\sum_{\lambda\in\Delta_+}g(\frac{\xi_{A|p}^k}k)[kf(\lambda/k)]\dim(V_\lambda)\\
=&\left(\sum_{(\lambda,t)\in\Delta_+(\mathcal L)}-\sum_{(\lambda,0)\in(\Delta_+\times\{O\})}\right)g(\frac{\xi_A(\lambda)}k)\dim(V_\lambda).
\end{align*}

We want to apply the Euler-Maclaurin formula of \cite[Section 4]{Guillemin-Sternberg-2007}. By simplicial division, we can divide $\Delta_+(\mathcal L)$ into a union of rational simplex. In fact, up to replace $L$ by $L^{r_0}$ for sufficiently divisible $r_0\in\mathbb N_+$, we may assume all simplexes are integral. Then we apply \cite[Theorem 4.2]{Guillemin-Sternberg-2007} on each simplexes and take sum. Before proceeding, let us fix some notations. Suppose that $Q$ is a full dimensional integral convex polytope in some lattice $\mathfrak M$. Denote by $\{F_A\}_{A=1}^{d_0}$ its facets. Suppose that $u_A$ is the primitive outer normal vector of $F_A$ and denote by $d\bar\sigma$ the measure on $\partial Q$ so that $d\bar\sigma|_{F_A}=\frac1{|u_A|}d\sigma_0$, where $d\sigma_0$ is the standard Lebesgue measure. Such a measure arises in counting the number of integral points in $kQ$,
$$\#\{kQ\cap\mathfrak M\}=k^{\dim(Q)}{\rm Vol}(Q)+\frac12k^{\dim(Q)-1}{\rm Vol}_{d\bar\sigma}|\partial Q|+O(k^{\dim(Q)-2}),~k\to+\infty.$$

We also need to deal with the $\dim(V_\lambda)$-terms. Recall \eqref{pi-weight}. We have
$$\rho-\frac12\kappa_P=\frac12(\sum_{\alpha\in\Phi_+^G,\alpha\not\perp\Delta_+}\alpha)\perp\Delta_+,$$ and
\begin{align*}
\langle\nabla\pi(\lambda),\rho-\frac12\kappa_P\rangle=\sum_{\alpha\in\Phi^G_+,\alpha\not\perp\Delta_+}(\prod_{\beta\not=\alpha,\beta\in\Phi_+^G,\beta\not\perp\Delta_+}\langle\beta,\lambda\rangle)\langle\alpha,\rho-\frac12\kappa_P\rangle.
\end{align*}
Consider $\Phi^L:=\Phi^G\cap\Delta_+^\perp$. Then $\Phi^L$ is a sub-root system of $\Phi^G$ and the Weyl group $W_L$ of $\Phi^L$ permutes $$\Phi^G_+\setminus\Phi^L=\{\alpha\in\Phi^G_+|\alpha\not\perp\Delta_+\}.$$
Hence $\nabla\pi(\lambda)$ is $W_L$-invariant. Choose $w_0\in W_L$ the longest element of $W_L$. Then
$$w_0(\rho-\frac12\kappa_P)=-(\rho-\frac12\kappa_P).$$
Hence $$\langle\nabla\pi(\lambda),\rho-\frac12\kappa_P\rangle=\langle w_0(\nabla\pi(\lambda)),w_0(\rho-\frac12\kappa_P)\rangle=0,$$
and
\begin{align*}\dim(V_\lambda)=C_{G/H}(\pi(\lambda)+\frac12\kappa_P(\nabla\pi(\lambda))+(\text{lower order terms})).
\end{align*}

Combining with the Euler-Maclaurin formula \cite[Theorem 4.2]{Guillemin-Sternberg-2007} (see also \cite{Kantor-Khovanskii-1993}),
\begin{align}
S^{(g)}_{1|k}(\mathcal X,\mathcal L)=&k^{n+1}\int_{\Delta_+}fg\pi d\lambda+\frac12k^n\int_{\partial\Delta_+}fg\pi d\bar\sigma+\frac12k^n\sum_{a-1}^{N_f}\int_{\mathcal F_a}g\pi d\bar\sigma\notag\\
&+\frac12k^n\int_{\Delta_+}fg\kappa_P(\nabla\pi)d\lambda-\frac12k^n\int_{\Delta_+}g\pi d\lambda+O(k^{n-1}),~k\to+\infty,\label{S-1k-compute}
\end{align}
where $\mathcal F_a=\{(\lambda,t)|t=f(\lambda),~\lambda\in\Omega_a\}$ is the facet of $\Delta_+(\mathcal L)$ that lies on the $a$-th piece of the graph of $f$.\footnote{Here we use the following identity in \cite[Theorem 4.2]{Guillemin-Sternberg-2007} (essentially in \cite[Eq. (3.15)]{Guillemin-Sternberg-2007}): Denote by $f(z)=z^n-1$ and $\omega=e^{\frac{2\pi}n\sqrt{-1}}$, then $\sum_{k=1}^{n-1}\frac1{1-\omega^k}=\frac{f''(1)}{2f'(1)}=\frac{n-1}2$. Thus the Todd functions  $\frac1n(\tau(s)+\sum_{k=1}^{n-1}\tau_{\omega^k}(s))=1+\frac12s+O(s^2),~s\to0$.}
Note that the primitive normal vector of $\mathcal F_a$ is $(m_a\Lambda_a,-m_a)$. We have
\begin{align*}
\int_{\mathcal F_a}g\pi d\bar\sigma&=\int_{\mathcal F_a}g\pi\frac1{m_a|(\Lambda_a,-1)|}d\sigma_0\\
&=\frac1{m_a}\int_{\Omega_a}g\pi d\lambda,~a=1,...,N_f.
\end{align*}
Thus
\begin{align}
S^{(g)}_{1|k}(\mathcal X,\mathcal L)=&k^{n+1}\int_{\Delta_+}fg\pi d\lambda+\frac12k^n\int_{\partial\Delta_+}fg\pi d\bar\sigma+\frac12k^n\int_{\Delta_+}fg\kappa_P(\nabla\pi)d\lambda\notag\\
&+\frac12k^n\sum_{a-1}^{N_f}(\frac1{m_a}-1)\int_{\Omega_a}g\pi d\lambda+O(k^{n-1}),~k\to+\infty,\label{S-1k-result}
\end{align}
Clearly,
\begin{align}\label{S-2k-result}
S^{(g)}_{2|k}(\mathcal X,\mathcal L)=&\frac12k^{n}\int_{\Delta_+}f\sum_{A=1}^r\xi_A(\lambda-\kappa_P)\frac{\partial g}{\partial\xi_A}(\xi_A(\lambda))\pi d\lambda+O(k^{n-1})\notag\\
=&\frac12k^{n}\int_{\Delta_+}f\langle\lambda-\kappa_P,\nabla g\rangle\pi d\lambda+O(k^{n-1}),~k\to+\infty.
\end{align}
Also, as in \eqref{S-1k-compute}
\begin{align}
{\rm h}^0(X,-kK_X)=&k^n\int_{\Delta_+}\pi d\lambda+\frac12k^{n-1}\int_{\partial \Delta_+}\pi d\bar\sigma\notag\\&+k^{n-1}\int_{\Delta_+}\langle\nabla\pi,\lambda\rangle d\lambda+O(k^{n-2}),~k\to+\infty.\label{h-0-result}
\end{align}


Plugging \eqref{S-1k-result}-\eqref{h-0-result} and the relation
$$\int_{\partial\Delta_+}\pi d\sigma=\int_{\partial \Delta_+}\langle \lambda,\nu\rangle d\sigma_0=n\int_{\Delta_+}\pi dy$$
into \eqref{S1-S2}, we get
\begin{align*}
2\frac{\int_{\Delta_+}\pi d\lambda}{\int_{\Delta_+}g\pi d\lambda}{\rm Fut}_g(\mathcal X,\mathcal L)=&-\int_{\partial\Delta_+}fg\pi d\bar\sigma+n\int_{\Delta_+}fg\pi d\lambda\notag\\
&-\int_{\Delta_+}f\kappa_P(\nabla \pi)g d\lambda+\int_{\Delta_+}f\langle\lambda-\kappa_P,\nabla g\rangle\pi d\lambda\notag\\
&+\sum_{a=1}^{N_f}(1-\frac1{m_a})\int_{\Omega_a}g\pi(\lambda)d\lambda,
\end{align*}
Note that for outer unit normal vector $\nu$ on each facet, $d\sigma=\langle\lambda,\nu\rangle d\sigma_0=d\bar\sigma+\langle\kappa_P,\nu\rangle d\sigma_0$. Taking integration by parts to the third term on the right-hand side
\begin{align*}
\int_{\Delta_+}\kappa_P f(\nabla \pi)g d\lambda=&\int_{\partial\Delta_+}f\langle\kappa_P,\nu\rangle\pi g d\sigma_0-\int_{\Delta_+}\kappa_P(\nabla f)g \pi d\lambda\\
&-\int_{\Delta_+} f\kappa_P(\nabla g)\pi d\lambda,
\end{align*}
we get the Proposition.
\end{proof}

Compare with \eqref{MNA-monomial}, for $\mathbb Q$-Fano spherical varieties we can strength Theorem \ref{K-stab-equiv} to the following:
\begin{coro}\label{FNA-prop}
Let $(X,L)$ be a $\mathbb Q$-Fano $G$-spherical variety which is locally $\mathbb Q$-factorial and $T\subset{\rm Aut}_G(X)$ be an $r$-dimensional torus. Suppose that $(\mathcal X,\mathcal L)$ the $G\times T$-equivariant test configuration of $(X,L)$ that is associated to the concave function $f$. Let $\Delta\subset\mathbb R^r$ be the moment polytope of the $T$-action with respect to the canonical lifting and $g\in C^\infty_0(\mathbb R^r)$. Then
\begin{align}\label{FNA}
\frac V{V_g}{\rm Fut}_g(\mathcal X,\mathcal L)={\rm M}_g^{\rm NA}(\mathcal X,\mathcal L)+\frac1{V_g}\sum_{a=1}^{N_f}(1-\frac1{m_a})\int_{\Omega_a}g\pi(\lambda)d\lambda\geq{\rm M}_g^{\rm NA}(\mathcal X,\mathcal L).
\end{align}
Consequently, $\frac V{V_g}{\rm Fut}_g(\mathcal X,\mathcal L)={\rm M}_g^{\rm NA}(\mathcal X,\mathcal L)$ if and only if $(\mathcal X,\mathcal L)$ has reduced central fibre.
\end{coro}
\begin{proof}The relation \eqref{FNA} can be proved in a same way as \eqref{MNA-monomial}. Note that
$$\int_{\Omega_a}g\pi(\lambda)d\lambda>0,~a=1,...,N_f.$$
The last point then follows from \eqref{X0}.
\end{proof}

\section{The stability criterion}

In this section we will prove a combinatorial criterion of $G$-uniformly $g$-modified stability.

To prove Theorem \ref{stab-criterion-Fano}, we need a technical lemma. Set
$$\mathcal C(\Delta_+)=\{f:\Delta_+\to\mathbb R|~\text{$f$ is concave and}~\nabla f\in\mathcal V(G/H)\}.$$
Define two functionals
\begin{align*}
{\mathcal D}_g(f):=&f(\kappa_P)-\frac1{V_g}\int_{\Delta_+}fg\pi\,dy,\\
{\mathcal J}_g(f):=&\max_{\Delta_+}f-\frac1{V_g}\int_{\Delta_+}fg\pi\,dy.
\end{align*}
By Proposition \ref{NA-functional}, for a $G$-equivariant normal test configuration $(\mathcal X,\mathcal L)$ of $(X,-K_X)$ is associated to $f$, it holds
$${\rm D}_g^{\rm NA}(\mathcal X,\mathcal L)=\mathcal D_g(f),~ {\rm J}_g^{\rm NA}(\mathcal X,\mathcal L)=\mathcal J_g(f).$$

\begin{lem}\label{uni-D-stab-lem}
Suppose that $X$ is a $\mathbb Q$-Fano spherical embedding of $G/H$ with moment polytope $\Delta_+$. Suppose that the barycenter satisfies \eqref{bar-condition}. Then there is a constant $\epsilon>0$ so that for any
$$f\in\hat{\mathcal C}(\Delta_+):=\mathcal C(\Delta_+)\cap\{{\rm pr}_{z}(\nabla f(\kappa_P))=0,\max f=0\},$$
it holds
\begin{align}\label{D-geq-eps-J}
{\mathcal D}_g(f)\geq\epsilon_0{\mathcal J}_g(f).
\end{align}
Here $${\rm pr}:\mathfrak N_\mathbb R(G/H)\to\mathcal V_{z\mathbb R}(G/H)(=\cap_{\alpha\in\Phi^{G/H}}\{\Lambda|\alpha(\Lambda)=0\})$$ is the projection with respect to the scalar product $\langle\cdot,\cdot\rangle.$
\end{lem}

\begin{proof}
Denote by ${\rm pr}_z^\perp={\rm Id}-{\rm pr}_z$.  Then
\begin{align*}
V_g\cdot\mathcal D_g(f)=&\int_{\Delta_+}(f(\kappa_P)-f)g\pi\,dy\notag\\
=&\int_{\Delta_+}(f(\kappa_p)-f+\nabla f(\kappa_P)(y-\kappa_P))g\pi\,dy\notag\\
&-\int_{\Delta_+}({\rm pr}_z(\nabla f(\kappa_P))(y-\kappa_P))g\pi\,dy\notag\\
&-\int_{\Delta_+}({\rm pr}_z^\perp(\nabla f(\kappa_P))(y-\kappa_P))g\pi\,dy
\end{align*}

By \eqref{bar-condition}, the second term
$$\int_{\Delta_+}({\rm pr}_z(\nabla f(\kappa_P))(y-\kappa_P))g\pi\,dy=V_g\cdot({\rm pr}_z(\nabla f(\kappa_P))({\mathbf b}(\Lambda_0)-\kappa_P))=0,$$
and the third term
\begin{align}\label{third-term}
-\int_{\Delta_+}({\rm pr}_z^\perp(\nabla f(\kappa_P))(y-\kappa_P))g\pi\,dy=-V_g\cdot({\rm pr}_z^\perp(\nabla f(\kappa_P))({\mathbf b}(\Lambda_0)-\kappa_P))\geq0,
\end{align}
since $-{\rm pr}_z^\perp(\nabla f(\kappa_P)\in(-\mathcal V(G/H))$.
By concavity
\begin{align}\label{first-term}
f(\kappa_p)-f+\nabla f(\kappa_P)(y-\kappa_P)\geq0.
\end{align}
Thus
\begin{align}\label{Df-non-neg}
V_g\cdot\mathcal D_g(f)\geq0,~\forall f\in\hat{\mathcal C}(\Delta_+).
\end{align}

Suppose that \eqref{D-geq-eps-J} is not true. 
There is a sequence $\{f_p\}_{p=1}^{+\infty}\subset\hat{\mathcal C}(\Delta_+)$ so that 
\begin{align}
V_g\cdot\mathcal J_g(f_p)=&\int_{\Delta_+}(-f_p)g\pi \,dy\equiv1,\label{Jf_p}\\
\lim_{p\to+\infty}\mathcal D_g(f_p)=&0.\label{Df_p}
\end{align}
By \eqref{Jf_p} and $f_p\leq0$, up to passing to a subsequence, $f_p$ converges locally uniformly to some $f_\infty\in\hat{\mathcal C}(\Delta_+)$. Combining with \eqref{Df-non-neg}, we have
$$0\leq\mathcal D_g(f_\infty)\leq\underline{\lim}_{p\to+\infty}\mathcal D_g(f_p)=0.$$
By \eqref{third-term}-\eqref{first-term}, we see that
$$f_\infty(y)=\zeta(y)+C$$
for some $\zeta\in\mathcal V_z(G/H)$ and $C\in\mathbb R$. But since $f_\infty\in\hat {\mathcal C}(\Delta_+)$,
$${\rm pr}_z(\nabla f_\infty(\kappa_P))=\zeta=0~\text{and}~C=0.$$
Hence $f_\infty\equiv0$.

Since
$$\lim_{p\to+\infty}f_p(\kappa_P)=f_\infty(\kappa_P)=0,$$
by \eqref{Jf_p},
$$V_g\cdot \mathcal D_g(f_p)=V_g\cdot f_p(\kappa_P)-\int_{\Delta_+}f_pg\pi\,dy\to1,~p\to+\infty.$$
A contradiction to \eqref{Df_p}. Hence \eqref{D-geq-eps-J} is true.
\end{proof}

\begin{proof}[Proof of Theorem \ref{stab-criterion-Fano}]
The direction (1)$\Rightarrow$(2): Note that if a $G$-equivariant normal test configuration $(\mathcal X,\mathcal L)$ of $(X,-K_X)$ is associated to $f$. Then the twist of $(\mathcal X,\mathcal L)$ by an $\exp(\Lambda)$-action with $\Lambda\in\mathcal V_{z\mathbb R}(G/H)\cong\mathfrak{aut}_G(X)$ is associated to $f_\Lambda(\lambda):=f(\lambda)-\Lambda(\lambda)$.
Also note that all the NA-functionals is invariant if we add to $f$ a constant. The direction is then a combination of \eqref{JNA-monomial} and Lemma \ref{uni-D-stab-lem}.

The direction (2)$\Rightarrow$(3) follows directly from Proposition \ref{Ding-Mabuchi};

The direction (3)$\Rightarrow$(4) is trivial;

The direction (4)$\Rightarrow$(1): Suppose that \eqref{bar-condition} fails. Then as in \cite[Lemma 2.4]{LL-IMRN-2021} one can construct a non-product $G$-equivariant normal test configuration $(\mathcal X,\mathcal L)$ of $(X,-K_X)$ so that $${\rm D}^{\rm NA}_g(\mathcal X,\mathcal L)={\rm M}^{\rm NA}_g(\mathcal X,\mathcal L)\leq0.$$
A contradiction to (4).
\end{proof}

\section{Appendix: Some Lemmas on $(X^{[\mathbf k]},L^{[\mathbf k]})$}

\subsection{The dimension of ${\rm H}^0(X^{[\mathbf k]},L^{[\mathbf k]})$}
\begin{lem}\label{dim-H0-Lk-lem}
Suppose that $(X,L)$ is polarized variety. Let $T\subset {\rm Aut}(M)$ be a torus that acts on $L$ through some lifting $\sigma$. Then on $(X^{[\mathbf k]},L^{\sigma[\mathbf k]})$ it holds
\begin{align}\label{dim-H0-Lk}
\dim{\rm H}^0(X^{[\mathbf k]},L^{\sigma[\mathbf k]})=\sum_{\lambda\in\mathfrak X(T)}\dim{\rm H}^0(X,L)^{(T)}_\lambda \prod_{A=1}^r\dim\mathcal O_{\mathbb P^{k_A}}(\lambda_A),
\end{align}
where $\mathbf\lambda=(\lambda_1,...,\lambda_r)\in\mathfrak X(T)\cong\mathbb Z^r$. In particular, if each $\lambda$ with ${\rm H}^0(X,L)^{(T)}_\lambda\not=O$ lies in $\mathbb N^r$, then
\begin{align}\label{dim-H0-Lk-ample}
\dim{\rm H}^0(X^{[\mathbf k]},L^{\sigma[\mathbf k]})=\sum_{\lambda\in\mathfrak X(T)}\dim{\rm H}^0(X,L)^{(T)}_\lambda C^{\mathbf k}_{\mathbf{k+\lambda}}.
\end{align}
\end{lem}

\begin{proof}
Since $T$ acts on $L$,
\begin{align*}
{\rm H}^0(X,L)=\bigoplus_{\lambda\in\mathfrak X(T)}{\rm H}^0(X,L)^{(T)}_\lambda.
\end{align*}
We fix a basis $\{s_{(\lambda)i}\}_{i=1}^{d_\lambda}$ of each ${\rm H}^0(X,L)^{(T)}_\lambda$.

On the other hand, suppose that $\bar s\in{\rm H}^0(X^{[\mathbf k]},L^{\sigma[\mathbf k]})$. Then the pull-back $\tilde s$ of $\bar s$ is a $T$-invariant section of $L\times(\mathbb C^{\mathbf k+\mathbf 1}\setminus\{O\})$. We can write
\begin{align*}
\tilde s=\sum_{\lambda\in\mathfrak X(T)}\sum_{i=1}^{d_\lambda}s_{(\lambda)i}f_{(\lambda)i},
\end{align*}
where each $f_{(\lambda)i}\in\mathbb C[\mathbb C^{\mathbf k+\mathbf 1}\setminus\{O\}]$. Since $\cup_{\lambda\in\mathfrak X(T)}\{s_{(\lambda)i}\}_{i=1}^{d_\lambda}$ is a basis of ${\rm H}^0(X,L)$, each $s_{(\lambda)i}f_{(\lambda)i}$ must be $T$-invariant. Consequently, $f_{(\lambda)i}$ descenders to a section $\bar f_{(\lambda)i}\in\otimes_{A=1}^r \mathcal O_{\mathbb P ^{k_A}}(\lambda_A)$. Hence we get \eqref{dim-H0-Lk}. \eqref{dim-H0-Lk-ample} follows from \eqref{dim-H0-Lk} and the identity (\cite[Chapter 2, Section 2.5]{Danilov-Itogi}),
\begin{align*}
\chi(\mathbb P^k,\mathcal O_{\mathbb P^k}(\lambda))=\frac1{k!}(\lambda+1)...(\lambda+k).
\end{align*}

\end{proof}

\subsection{The canonical divisor of $X^{[\mathbf k]}$}
\begin{lem}\label{-K_{X^k}-lem}
Suppose that $X$ is $\mathbb Q$-Fano variety. Let $T\subset {\rm Aut}(M)$ be a torus that acts on $-K_X$ through the canonical lifting $\sigma_0$. Then it holds
\begin{align*}
-K_{X^{[\mathbf k]}}=(-K_X)^{\sigma_0[\mathbf k]}+\sum_{A=1}^r{\rm pr}_A^*\mathcal O_{\mathbb P^{k_A}}(k_A+1).
\end{align*}
on the variety $(X^{[\mathbf k]},(-K_X)^{\sigma_0[\mathbf k]})$.
\end{lem}

\begin{proof}
Without loss of generality we may assume that $-K_X$ is Cartier. Otherwise we replace $-K_X$ by some $-m_0K_X, ~m_0\in\mathbb N_+$. As $T$ acts on ${\rm H}^0(X,-K_X)$, there is a section $s_0\in{\rm H}^0(X,-K_X)_{\chi_0}^{(T)}$. That is,
$$t\cdot s(t^{-1}x)=t^{\chi_0}s(x),~\forall x\in X~\text{and}~t\in T.$$

By definition, the pull back of any section $\bar s\in{\rm H}^0(X^{[\mathbf k]},(-K_X)^{\sigma_0[\mathbf k]})$ on $X\times\mathbb C^{\mathbf k+\mathbf 1}$ is a $T$-invariant section of $-K_X\times\mathbb C^{\mathbf k+\mathbf 1}$, and any such section descend to a section in ${\rm H}^0(X^{[\mathbf k]},(-K_X)^{\sigma_0[\mathbf k]})$. Set $\tilde s_0=s_0\times\mathbb C^{\mathbf k+\mathbf 1}$. Then any $$\tilde s_{0,\mu}:=\tilde s_0\prod_{A=1}^r(z^{(A),0})^{\mu_{A,0}}...(z^{(A),k_A})^{\mu_{A,k_A}}$$
with $\mu_{A,0}+...+\mu_{A,k_A}=\chi_{0A}$ for each $A=1,...,r$ is a $T$-invariant section of $-K_X\times\mathbb C^{\mathbf k+\mathbf 1}$, that descends to a section $s_{0,\mu}$ of $(-K_X)^{\sigma_0[\mathbf k]}$. Clearly, the divisor of $\tilde s_{0,\mu}$ is $T$-invariant and descends to the divisor $\mathfrak d_{s_{0,\mu}}$ of $s_{0,\mu}$. Denote by $\mathfrak d_{s_0}\subset X$ the divisor of $s_0$ in $X$, which is also $T$-invariant, we have
\begin{align}\label{div-(-K_X)^k}
\mathfrak d_{s_{0,\mu}}=\mathfrak d_{s_0}^{[\mathbf k]}+\sum_{A=1}^r{\rm pr}_A^*\mathcal O_{\mathbb P^{k_A}}(\chi_{0A}).
\end{align}

On the other hand, recall the local coordinate charts $\mathcal U_{[i]}$ with each $i_A\in\{0,...,k_A\}$ constructed in Section 3.2, and the local coordinates $(x_{[i]},\zeta_{[i]})$ on it. It is direct to check that on each $\mathcal U_{[i]}\cap U_{[j]}$,
$$x_{[i]}=(\zeta_{[i]}^j)\cdot x_{[j]},~\zeta_{[j]}=\zeta_{[j]}(\zeta_{[i]}).$$
In particular, $\zeta_{[j]}$ is a function that depends only on $\zeta_{[i]}$. Thus under the local section $$e_{[i]}=(dx_{[i]}\wedge d\zeta_{[i]}=dx_{[1]}^1\wedge...\wedge dx_{[i]}^n)\bigwedge\wedge _{A=1}^r(d\zeta_{[i]}^{(A),1}\wedge...\wedge d\zeta^{(A),k_A})$$
of $-K_{X^{[\mathbf k]}}$, the transition is given by
$$e_{[i]}=\det(\frac{\partial x_{[i]}}{\partial x_{[j]}})\det(\frac{\partial \zeta_{[i]}}{\partial \zeta_{[j]}})e_{[j]}.$$
Clearly, the induced action of $T$ on each $e_{[i]}$ gives the canonical lifting $\sigma_0$. Hence, the pull-back of a section $s'_0$ of $-K_{X^{[\mathbf k]}}$ on $X\times(\mathbb C^{\mathbf k+\mathbf 1}\setminus\{O\})$ is a $T$-invariant section of $-K_{X\times(\mathbb C^{\mathbf k+\mathbf 1}\setminus\{O\})}$. As before, from $s_0$ we can construct a divisor
\begin{align*}
\mathfrak d_{s'_{0,\mu}}=\mathfrak d_{s_0}^{[\mathbf k]}+\sum_{A=1}^r{\rm pr}_A^*\mathcal O_{\mathbb P^{k_A}}(\chi_{0A})+\sum_{A=1}^r{\rm pr}_A^*\mathcal O_{\mathbb P^{k_A}}(k_{A}+1)
\end{align*}
of $-K_{X^{[\mathbf k]}}$. Combining with \eqref{div-(-K_X)^k} we get the Lemma.

\end{proof}

\begin{rem}
Let $X$ be a $\mathbb Q$-Fano $G$-spherical variety. Then there always exists a $B$-semiinvariant section $s$ of $-K_X$ with divisor ${\mathfrak d}$ given by \eqref{weil-div} and $B$-weight $\kappa_P$ (cf. \cite[Theorem 1.2]{Ga-Ho}). By \eqref{chi-0-can}, $s_0$ is $T$-invariant under the canonical lifting of the $T$-action and we get a divisor
$$\overline{\mathfrak d}={\mathfrak d}^{[\mathbf k]}+\sum_{A=1}^r{\rm pr}_A^*\mathcal O_{\mathbb P^{k_A}}(k_{A}+1)$$
of $-K_{X^{[\mathbf k]}}$. This is precisely the divisor we get when applying \cite[Theorem 1.2]{Ga-Ho} to the variety $X^{[\mathbf k]}$.
\end{rem}

\subsection{An intersection formula}

\begin{lem}\label{intersection-integral}
Let $X$ be a projective variety with effective $T$-action and $L$ a $T$-linearized ample line bundle on it. Let $\Delta_+(X,L)$ be the moment polytope and $g$ a monomial \eqref{monom-g} on it. Denote by ${\rm DH}_T(X,L)$ the corresponding Duistermaat-Heckman measure. Then
$$\frac{\mathbf k!}{(n+|\mathbf k|)!}(L^{[\mathbf k]})^{n+|\mathbf k|}=\int_{\Delta_+(X,L)}g(\lambda){\rm DH}_T(X,L)(\lambda).$$
\end{lem}

\begin{proof}
Consider a $T$-invariant resolution $\varphi:\tilde X\to X$ and $\tilde L:=\varphi^*L$ on it. Let $\omega\in2\pi c_1(L)$ be a K\"ahler form on $X$. Then $\tilde\omega:=\varphi^*\omega$ is a semi-positive closed $(1,1)$-form on $\tilde X$. The moment maps of the $T$-actions on $(X,L,\omega)$ and $(\tilde X,\tilde L, \tilde\omega)$ are related by:
$$\tilde {\mathbf m}_{\tilde \omega}={\mathbf m}_\omega\circ\varphi.$$
As a consequence, $\Delta_+(X, L)=\Delta_+(\tilde X,\tilde L)$ and
$${\rm DH}_T(X,L)={\rm DH}_T(\tilde X,\tilde L).$$
It is well-known that both $\Delta_+(\tilde X,\tilde L)$ and ${\rm DH}_T(\tilde X,\tilde L)$ are independent with the choice of moment map. Thus by choosing a smooth $\tilde\omega'\in2\pi c_1(\tilde L)$ and using $\tilde{\mathbf m}_{\tilde\omega'}$, it follows from \cite[Section 2]{Han-Li-KRS} that
\begin{align*}
\frac{\mathbf k!}{(n+|\mathbf k|)!}(\tilde L^{[\mathbf k]})^{n+|\mathbf k|}=&\int_{\Delta_+(\tilde X,\tilde L)}g(\lambda){\rm DH}_T(\tilde X,\tilde L)(\lambda)\\
=&\int_{\Delta_+(X,L)}g(\lambda){\rm DH}_T(X,L)(\lambda).
\end{align*}
On the other hand, $\varphi$ is a generically finite surjective proper map and $\deg(\varphi)=1$. Hence $(L^{[\mathbf k]})^{n+|\mathbf k|}=(\tilde L^{[\mathbf k]})^{n+|\mathbf k|}$ and we get the Lemma.
\end{proof}

\vskip25pt


\end{document}